\documentclass[11pt]{amsart}
\usepackage{amsmath,amsfonts,amsthm,amssymb,amscd,graphicx}
\usepackage{hyperref}
\usepackage{cancel}
\usepackage[normalem]{ulem}
\usepackage{color}
 \usepackage[full]{textcomp}
     \usepackage[osf]{newtxtext} 
     \usepackage{cabin} 
     \usepackage[varqu,varl]{inconsolata} 
     \usepackage[bigdelims,vvarbb]{newtxmath} 
     \usepackage[cal=boondoxo]{mathalfa} 

\newtheorem{thm}{Theorem}[section]
\newtheorem*{*thm}{Theorem}
\newtheorem{lemma}[thm]{Lemma}
\newtheorem*{lemma*}{Lemma}
\newtheorem{prop}[thm]{Proposition}
\newtheorem*{prop*}{Proposition}

\newtheorem{corr}[thm]{Corollary}
\newtheorem*{corr*}{Corollary}
\theoremstyle{definition}
\newtheorem{dfn}[thm]{Definition}
\newtheorem{exmple}[thm]{Example}
\newtheorem{exmples}[thm]{Examples}

\theoremstyle{remark}

\newtheorem*{rmq}{\textit{Remark}}
\newtheorem{rmk}[thm]{\textit{Remark}}
\renewcommand{\proof}{\noindent\textit{Proof}\/: \,\,}

\newcommand{\C}{{\mathbb{C}}}

\newcommand{\R}{{\mathbb{R}}}

\newcommand{\Z}{{\mathbb{Z}}}

\newcommand{\bP}{{\mathbb{P}}}

\newcommand\germ[1]{{\mathfrak{#1}}}
\def\eun{\germ n}
\def\eq{\germ q}
\newcommand{\eh}{\germ h}%
\newcommand{\eg}{\germ g}%
\def\el{\germ l}
\newcommand{\ez}{\mathfrak{z}}%

\newcommand\FF{{\mathcal  F}}

\newcommand\MM{{\mathcal M}}
\newcommand\II{{\mathcal I}}
\newcommand\EE{{\mathcal E}}
\newcommand\UU{{\mathcal U}}

\newcommand\VV{{\mathcal V}}

\newcommand{\comp}{\raise1pt\hbox{{$\scriptscriptstyle\circ$}}}
\def\lset{\left\{ }  
\def\rset{\right\} }  
\def\set#1{\lset#1\rset}
\def\st{\mid}   
\def\sett#1#2{\lset #1 \st #2 \rset}
\def\del{\partial}
\def\tr{\mathop{\rm Tr}\nolimits}
\newcommand\Tr{{}^{\mathsf{T}}\kern-0.9pt} 

\def\mapright#1{\mathop{\vbox{\ialign{
                                ##\crcr
    ${\scriptstyle\hfil\;\;#1\;\;\hfil}$\crcr
 \noalign{\kern2pt\nointerlineskip}
    \rightarrowfill\crcr}}\;}}
\def\mapleft#1{\mathop{\vbox{\ialign{
                                ##\crcr
    ${\scriptstyle\hfil\;\;#1\;\;\hfil}$\crcr
 \noalign{\kern2pt\nointerlineskip}
    \leftarrowfill\crcr}}\;}}
\newcommand\rarrow[3]{\smash{\mathop{\hbox to#3{\rightarrowfill}}\limits
^{\scriptstyle#1}_{\scriptstyle#2}}}
\newcommand\larrow[3]{\smash{\mathop{\hbox to#3{\leftarrowfill}}\limits
^{\scriptstyle#1}_{\scriptstyle#2}}}
\def\into{\hookrightarrow}

\renewcommand\setminus{-}
\def\ii{{\rm i}}

\newcommand\gl[1]{\operatorname{GL}({#1})}

\newcommand\ogr[1]{\operatorname{O}({#1})}
\newcommand\ugr[1]{\operatorname{U}({#1})}
\newcommand\slgr[1]{\operatorname{SL}({#1})}
\newcommand\smpl[1]{\operatorname{Sp}({#1})}

\newcommand\im{\operatorname{Im}}
\newcommand\re{\operatorname{Re}}

\newcommand\gr{\operatorname{Gr}}

\newcommand\Hom{\mathop{\rm Hom}\nolimits}
\newcommand\End{\mathop{\rm End}\nolimits}
\def\Ext{\mathop{\rm Ext}\nolimits}
\def\VMHS{\text{VMHS}}
\def\vrulecup{\kern1.30pt\cup\kern-4.15pt\vrule height 6.0pt depth 0pt}
\newcommand\Ad[1]{\operatorname{Ad}{#1}\,}
\newcommand\ad[1]{\operatorname{ad}{#1}\,}
\newcommand\Lie[1]{\operatorname{Lie}(#1)}
\def\half{\frac 12}

\def\changed#1#2{%
\textcolor{black}{#1}\textcolor{black}{#2}%
}

\def\pd{\partial}
\def\piq{\pi_{\eq}}
\def\pip{\pi_+}
\def\deriv#1#2{\frac{\pd #1}{\pd #2}}

\begin{document}
\title{Differential Geometry of the Mixed Hodge Metric}
\author{Gregory Pearlstein}
\address{Department of Mathematics\\ Mailstop 3368, Texas A \&M University \\
College Station, TX 77843-3368, USA.}
\email{gpearl@math.tamu.edu}
\author{Chris  Peters}
\address{Faculteit Wiskunde en Informatica \\
Section DM 6th floor MF \\
Postbus 513, MetaForum \\
5600 MB Eindhoven, Netherlands.}
\email{c.a.m.peters@tue.nl}
\maketitle

{\small
\noindent \textbf{Summary:} We investigate properties of the mixed Hodge metric of a mixed period domain. In particular, we calculate its curvature and the curvature of the Hodge bundles. We also consider when the pull back metric via a period map is K\"ahler. Several applications in cases of geometric interest are given, such as for normal functions and biextension bundles.
}

\section{Introduction}  \label{sec:intro}
\subsection{Overview} Let $f:X\to S$ be a smooth, proper morphism between
complex algebraic varieties.  Then, by the work of Griffiths~\cite{periods},
the associated local system $\mathcal H_{\mathbb Q} = R^k {f^*}\mathbb Q_X$
underlies a variation of pure Hodge structure of weight $k$, which
can be described by a \emph{period map}
\begin{equation}
     \varphi:S\to\Gamma\backslash D,                \label{eqn:gr-1}
\end{equation}
where $\Gamma$ is the \emph{monodromy group} of the family.
In the case where the morphism $X\to S$ is no longer smooth and proper the
resulting local system underlies a variation of (graded-polarized) mixed Hodge
structure over a Zariski open subset of $S$~\cite{sz}. As in the pure case
considered by Griffiths, a variation of mixed Hodge structure can be described
in terms of a period map which is formally analogous to \eqref{eqn:gr-1} except
that $D$ is now a classifying space of graded polarized mixed Hodge
structure~\cite{higgs,usui}.

\par As we shall explain below, there is a natural metric
on such $D$,  induced by the \emph{mixed Hodge metric}  \eqref{eqn:grhmetric}. Deligne's second order calculations
involving this metric in the pure case \cite{deligne} can be extended to the
mixed setting, as we show in this article.
For instance, we find criteria as to when the induced
Hodge metric on $S$ is K\"ahler. We also compute the curvature tensor of
this metric, with special emphasis on cases of  interest in the study of
algebraic cycles, archimedean heights and iterated integrals.
The alternative approach   \cite[Chap. 12]{periodbook} 
in the pure case based on the Maurer-Cartan form
does not seem to generalize as we encounter incompatibilities between
the metric and the complex structure as demonstrated in \S~\ref{sec:reductive}.


\subsection{The Pure Case}  Returning to the pure case, we recall that $D$ parametrizes Hodge
structures of weight $k$ on a reference fiber $H_{\mathbb Q}$ of $\mathcal H_{\mathbb Q}$
with given Hodge numbers $\{h^{p,q}\}$ and  polarized  by a non-degenerate
bilinear form $Q$ of parity $(-1)^k$.  The monodromy group $\Gamma$ is
contained in the real Lie group $G_\R\subset GL(H_\R)$ of automorphisms of
the polarization $Q$.

\par In terms of differential geometry, the first key fact is that
$G_\R$ acts transitively on $D$ with compact isotropy, and hence
$D$ carries a $G_\R$ invariant metric.   It is induced
by the polarizing  form $Q$ as follows.  Any   Hodge
filtration $F$ on $H_\C$  induces
\begin{equation}   \label{eqn:hmetric}
h_F(x,y) := Q(C_Fx,\bar y), x,y\in H_\C,
\end{equation}
where $C_F|H^{p,q}= \ii^{p-q}$ is the Weil-operator. This
is a metric as a consequence of  the two Riemann bilinear relations: the first, $
        Q(F^p,F^{k-p+1}) = 0                        $ states that the Hodge decomposition is $h_F$-orthogonal and
        the second states that $h_F$ is a metric on each Hodge-component.

Next, by describing the Hodge
structures parameterized by $D$ in terms of the corresponding flags
$$
       F^p H_{\C} = \bigoplus_{a\geq p}\, H^{a,k-a}
$$
we obtain an open embedding of $D$ into the flag manifold $\check{D}$
consisting of decreasing filtrations $F^* H_\C$ such that
$\dim F^p = \sum_{a\geq p}\, h^{a,k-a}$ which satisfy only   the first Riemann
bilinear relation.
In particular, via this embedding, the set $D$ inherits the structure of a
complex manifold upon which the group $G_\R$ acts via biholomorphisms.

\par As a flag manifold, the tangent space at $F$ to $\check{D}$ can be
identified with a subspace of
\begin{equation}
       \bigoplus_p\, \text{Hom}(F^p,H_\C/F^p).         \label{eqn:flag-tangnet}
\end{equation}
Via this identification, we say that a tangent vector is
\emph{(strictly) horizontal} if it is contained in the subspace
$$
       \bigoplus_p\, \text{Hom}(F^p, F^{p-1}/F^p).
$$
One of the basic results of~\cite{periods} is that the period map
associated to a smooth proper morphism $X\to S$ as above is holomorphic,
horizontal and locally liftable.

\par Combining the previous two paragraphs, the metric \eqref{eqn:hmetric} on
$V$ induces a functorial metric on \eqref{eqn:flag-tangnet} and hence induces
a hermitian metric $h$ on the analytic open subset $D$ of the smooth variety
$\check D$.  In particular, since $G_{\mathbb R}$ acts transitively on
$D$ via biholomorphisms and
$$
        h_{g.F}(x,y) = h(g^{-1}x,g^{-1}y)
$$
for all $g\in G_{\mathbb R}$ and $F\in D$, it follows that $h$ is a
$G_{\mathbb R}$-invariant metric on $D$.

\par By \cite[Theorem 9.1]{curv} the holomorphic sectional curvature of $D$
along horizontal tangents is negative and bounded away from zero.  In
particular, as a consequence of this curvature estimate, if $S\subset\bar S$
is a smooth normal crossing compactification with unipotent monodromy near
$p\in\bar S -S$, then by~\cite{schmid} the period map $\varphi$ has at worst logarithmic
singularities near $p$.


\subsection{Mixed Domains}  In the mixed case, period maps of geometric origin are holomorphic and
satisfy the analogous horizontality condition (\cite{usui,sz}).  However,
although there is a \emph{natural Lie group $G$} (see \S~\ref{ssec:homstr}) which acts transitively on the
classifying spaces of graded-polarized mixed Hodge structure, the isotropy
group is no longer compact, and hence there is no $G$-invariant hermitian
structure.  In spite of this, A. Kaplan observed in~\cite{kaplan} that one
could construct a natural hermitian metric on $D$ in the mixed case which was
invariant under a pair of subgroups $G_\R$ and $\exp(\Lambda)$ of $G$ which
taken together act transitively on $D$.  The subgroup $\exp(\Lambda)$ (see \S~\ref{ssec:tang}) depends
upon a choice of base point in $D$ and intersects the group $G_\R$
non-trivially.  Nonetheless, as we said before, by emulating the computations of Deligne
in~\cite{deligne}, we are able to compute the curvature tensor of $D$
in the mixed case (cf. \S \ref{sec:curv}).

\par Let us elaborate on this by defining  the natural metric.
A mixed Hodge structure $(F,W)$ on $V$ induces a
unique functorial bigrading~\cite{tdh}, the \emph{Deligne splitting}
\begin{equation}
      V_\C = \bigoplus_{p,q}\, I^{p,q} \label{eqn:DelSplit}
\end{equation}
such that $F^p = \bigoplus_{a\geq p}\, I^{a,b}$,
$W_k = \bigoplus_{a+b\leq k}\, I^{a,b}$  and
\[
\bar I^{p,q} = I^{q,p} \mod \bigoplus_{a<q,b<p}\, I^{a,b}.
\]
In the pure case a polarization induces a hermitian inner product
for which the Hodge decomposition is orthogonal. In the mixed case
we first declare the splitting \eqref{eqn:DelSplit} to be  $h_{(F,W)}$--orthogonal and then define the metric
on $I^{p,q}$ making use of the graded polarization $(\gr h)_F$ as follows.
The summand $I^{p,q}$ maps isomorphically onto the subspace
$H^{p,q}$ of $\gr^W_{p+q}$. So  on classes $[z]$   of elements    $z\in I^{p,q}\subset W^{p+q}$
modulo $W^{p+q-1}$ the metric  $h_{F,W}$ can be defined by setting:
\begin{equation} \label{eqn:grhmetric}
h_{(F,W)}(x,y) = (\gr h)_F ([x],[y]) ,\quad x,y \in I^{p,q}.
\end{equation}
This is the \emph{mixed Hodge metric} alluded to previously.
By functoriality it induces
Hodge metrics on $\End(V)$ (see \eqref{eqn:mhm})  and hence also on the Lie algebra of $G$. As in the pure case this induces a natural metric on the mixed period domain (see Definition~\ref{HMonD}).  It is
these metrics that form our principal  subject of investigation of this paper.

\begin{rmk} A Mumford--Tate domain $D_M$ classifies pure Hodge structures with extra Hodge tensors~\cite{ggk}.
In analogy with the classifying spaces of pure Hodge structures, $D_M$ is the orbit of a generic point $F\in D_M$
under the real points Mumford--Tate group of $F$.  The analog for mixed Hodge structures are mixed Mumford--Tate
domains, e.g. the mixed Shimura varieties of Pink and Milne.  See Remark~\ref{mMTD}. All of the Lie theoretic calculations done in
section~\ref{sec:class}, and hence all of the applications in the subsequent sections remain true for mixed Mumford--Tate
domains. \end{rmk}

\subsection{Examples} \label{ssec:exmples}
To get an idea of the nature of these metrics in the mixed situation we
give a few examples.
\begin{enumerate}
\item  Consider
the mixed Hodge structure on the cohomology of quasi-projective curves.
So, let $X  $ be a compact Riemann surface of genus $g$ and $S$ be
a finite set of points on $X$.  Then, $F^1 H^1(X \setminus S,\C)$
consists of holomorphic 1-forms $\Omega$ on $X\setminus S$ with at worst
simple poles along $S$, and the mixed Hodge metric is given by
\begin{equation}
        ||\Omega||^2 = 4\pi^2\sum_{p\in S}\, |\text{Res}_p(\Omega)|^2
          + \sum_{j=1}^g\, \left|\int_X\, \Omega\wedge\bar\varphi_j\right|^2,
        \label{eqn:rs-mhm}
\end{equation}
where $\{\varphi_j\}$ is unitary frame for $H^{1,0}(X)$ with respect to the
standard Hodge metric on $H^1(X,\C)$.

\par To verify this, we recall that in terms of Green's functions, the
subspace $I^{1,1}$ can be described as follows: If $H$ is the space of
real-valued harmonic functions on  $X\setminus S$ with at worst
logarithmic singularities near the points of $S$, then
\begin{equation}
      I^{1,1}\cap H^1(X\setminus S,\R)
       = \sett{\sqrt{-1}\cdot \del (f)}{  f\in H } .    \label{eqn:i11}
\end{equation}
Indeed, the elements of $I^{1,1}$ will be meromorphic 1-forms with simple
poles along $S$.  The elements $\sqrt{-1}\cdot \del (f)$ are also real cohomology
classes since the imaginary part is exact.

Direct calculation using\eqref{eqn:i11}  and Stokes'  theorem shows that $I^{1,1}$ consists
of the elements in $F^1$ which pair to zero against $H^{0,1}$.
Therefore, the terms $\int_X\,\Omega\wedge\bar\varphi_j$ appearing in
\eqref{eqn:rs-mhm} only compute the Hodge inner product for the component
of $\Omega$ in $I^{1,0}$.

\item Recall that the dilogarithm \cite[\S 1]{hain}  is  the \changed{}{iterated integral}
\[
\changed{}{\text{Li}}_2 (x)=\int_0^x  w_1\cdot  w_2,\quad  w_1= \frac{1}{2\pi \ii} \cdot \frac {dz}{1-z},\quad
w_2=\frac{1}{2\pi \ii}\cdot  \frac {dz}{z}.
\]
For the corresponding   variation of mixed Hodge structure arising
from the mixed Hodge structure on $\pi_1(\mathbb P^1-\{0,1,\infty\} ,\changed{}{z})$, the pull back
metric is given by
\begin{equation}
         \|\nabla_{d/dz}\|^2 =   \left[\frac{1}{|z|^2} + \frac{1}{|z-1|^2}  \right].
         \label{eqn:dilog-ex}
\end{equation}
For a proof, we refer to \S~\ref{defoinAb}.

\item Consider mixed Hodge structures  whose Hodge numbers are $h^{0,0} = h^{-1,-1} =1$.
The corresponding classifying space is isomorphic to $\C$ with the Euclidean
metric. In particular, the curvature is identically zero.  Note that the corresponding extensions are
parametrized by $\Ext^1_{\rm MHS} (\Z(0),\Z(1))=\C^*$: these are equivalence classes of mixed
Hodge structures, but we are not considering these.

\item  Let $(X,\omega)$ be a compact K\"ahler manifold of dimension $n$, and
$(F,W)$ denote the mixed Hodge structure on $V = \bigoplus_p\, H^p(X,\mathbb C)$
defined by setting $I^{p,q} = H^{n-p,n-q}(X)$.  For any $u\in H^{1,1}(X)$
let $N(u)$ denote the linear map on $V$ defined by
\begin{equation}
           N(u)v = u\wedge v                            \label{eqn:pmhs}
\end{equation}
Then, $N(u)$ is of type  $(-1,-1)$ with respect to $(F,W)$. By the Hard Lefschetz
theorem, if $u$ is a K\"ahler class the intersection pairing on $X$ can be used
to construct a graded-polarization of $(F,W)$.  In the language of~\cite{schmid,degeneration}
$(F,W)$ is an example of a mixed Hodge structure polarized by $N$.
\changed{}{ See also Section~\ref{ssec:results}, item 5.}


\item The period domain quotients $\Gamma\backslash D$ and their Mumford--Tate
domain analogs can be partially compactified by adjoining boundary components
consisting of nilpotent orbits~\cite{ku}.  Via the theory of polarized mixed Hodge
structures, such boundary components acquire mixed Hodge metrics.



\end{enumerate}

\par  Some properties the Hodge metric has in the pure case are no
longer valid in the mixed situation. This is already clear from Example~3:
we can not expect $D$ to have holomorphic sectional curvature which is
negative and bounded away from zero along horizontal directions.
Nonetheless, period maps of variations of mixed Hodge structure of geometric
origin satisfy a system of admissibility conditions which ensure that they
have good asymptotic behavior.  At the level of $D$-modules, this is
exemplified by Saito's theory of mixed Hodge modules.  At the level of
classifying spaces, one has the analogs of  Schmid's nilpotent orbit
theorem~\cite{dmj,nilp} and the SL$_2$-orbit theorem~\cite{knu,sl2anddeg}.
\changed{}{
\begin{rmq}  The metric in \cite{nilp}  is obtained by twisting the
metric considered in this paper by a factor which measures
how far a point in $D$ is from being an $\mathbb R$-split
mixed Hodge structure.
\end{rmq}
}

\subsection{Results} \label{ssec:results}
\begin{enumerate}
\item A mixed period domain $D$ is an open subset of a homogeneous space for a
complex Lie group $G_{\C}$, and hence we can identify $T_F(D)$ with a
choice~\eqref{eqn:tang-alg} of complement $\eq$ to the stabilizer of $F$ in
$\Lie{G_{\C}}$.  In analogy with Th\'eor\`eme $(5.16)$ of \cite{deligne},
the holomorphic sectional curvature in the direction
$u\in\eq\simeq T^{1,0}_F(D)$ is given by (cf. Theorem~\eqref{thm:hol-curv}):
$$
\aligned
   R_\nabla (u,\bar u) =
   &- [(\ad{{\bar u }_+^*})_\eq ,(\ad{{\bar u }_+})_\eq]
   -  \ad{[ u, \bar u]_0} \\
   &- \left(\ad{([u,\bar u]_+ + [u,\bar u]_+^*)} \right)_\eq
\endaligned
$$
where the subscripts ${}_{\eq}$, ${}_0$, ${}_+$  denote projections onto various
subalgebras of $\Lie{G_{\C}}$, and $*$ is adjoint with respect to the
mixed Hodge metric;  {the adjoint operation is meant to be preceded by the
projection operator  ${}_+$.}

\item 
In the pure case it is well known~\cite[Prop. 7.7]{periods2} that the
``top'' Hodge bundle\footnote{In standard  notation; it  differs from the
notation employed in \cite{periods2}. } $\FF^n$ is positive in the differential
geometric sense  while the ``dual'' bundle  $\FF^0/\FF^1$ is
negative. In the mixed setting, the Chern form of the top Hodge bundle is
non-negative, and positive wherever the $(-1,1)$-component of the derivative
of the period map acts non-trivially on the top Hodge bundle.
See Corollary \ref{corr:top-hodge}.


\item \changed{}{If we pull back the Hodge metric
via the  period map, we get a pseudo-metric in the sense that
at each tangent space it induces  a semi-positive scalar product: non-trivial directions
may have zero ``length''.
By~\cite{Zhiqin}, in the pure setting,
 the pseudo-metric is K\"ahler, that is, it is  hermitian and its associated
 $(1,1)$-form is closed.  So} it is a natural question to ask when there are more instances
where the  pullback of the mixed Hodge metric
along a mixed period map is K\"ahler.  In \S \ref{sec:kaehler}, we answer
this question in terms of a system of partial differential equations;
in particular we prove:
\begin{*thm}[c.f. Theorem~\ref{thm:kahler2}] Let $\VV$ be a variation of mixed
Hodge structure with only two non-trivial weight graded-quotients
$\gr^W_a$ and $\gr^W_b$ which are adjacent, i.e. $|a-b|=1$.  Then,
the pullback of the mixed Hodge metric along the period map of
$\VV$ is \changed{}{a K\"ahler pseudo-metric}.
\end{*thm}

  An example (cf. \S \ref{defoinAb}) of a variation of mixed Hodge
structure of the type described at the end of the previous paragraph arises in
\emph{homotopy theory} as follows:  Let $X$ be a smooth complex projective
variety and $J_x$ be the kernel of the natural ring homomorphism
$\Z\pi_1(X,x)\to\Z$.  Then, the stalks $J_x/J_x^3$ underlie a variation of
mixed Hodge structure with weights 1 and 2 and constant graded Hodge
structure~\cite{hain}.
We show:
 \begin{prop*}[c.f. Corollaries~\ref{corr:fund-group},  \ref{corr:kahler}]  If
the differential of the period map of $J_x/J_x^3$ is injective for a smooth
complex projective variety $X$ then the pull back metric is K\"ahler and its
holomorphic sectional curvature of  is non-positive.
 \end{prop*}
Concerning the injectivity hypothesis, which   is directly related to mixed
Torelli theorems we note that these hold for compact curves \cite{hain} as
well as once punctured curves \cite{kaenders}.

\item The curvature of a \emph{Hodge--Tate domain} is identically zero:
\begin{prop*}[c.f.~Lemma~\ref{lemma:tate-case} and Corollary~\ref{corr:kahler}]
Suppose  $h^{p,q}=0$
unless $p=q$. Then the curvature of the mixed Hodge metric
is identically zero, and pulls back to a K\"ahler pseudo-metric along
any period map $\varphi:S\to\Gamma\backslash D$.
\end{prop*}
Consequently, a necessary condition for a period map
$\varphi:S\to\Gamma\backslash D$ of Hodge-Tate type to have injective
differential is that $S$ support a K\"ahler metric of holomorphic sectional
curvature $\leq 0$.  Important examples of such variations arise in the study
of mixed Tate motives and polylogarithms~\cite{3pts} and mirror
symmetry~\cite{d2m}.


\item  Let $X\to \Delta^r$ be a holomorphic family of compact K\"ahler manifolds
of dimension $n$ equipped with a choice of K\"ahler class common to every member
of the family.   Let $(F(s),W)$ be the corresponding variation of mixed Hodge
structure defined by setting $I^{p,q} = H^{n-p,n-q}(X_s)$ as in
Subsection~\ref{ssec:exmples}.4.
Suppose that $\lambda_1,\dots,\lambda_k\in H^{1,1}(X_s,\mathbb R)$ for all $s$ (e.g. a
set of K\"ahler classes common to all members of the family).  Let $L_{\mathbb C}$ be the complex linear
span of $\lambda_1,\dots,\lambda_r$ and let   $u:\Delta^r\to L_{\mathbb C}$ be a
holomorphic function.  Then, with $N$ as in \eqref{eqn:pmhs}
\begin{equation}
      (e^{\ii N(u(s))} \cdot F(s),W),           \label{eqn:kahler-def}
\end{equation}
is a variation of mixed Hodge structure.  The curvature  of the corresponding
classifying space is  semi-negative  along directions tangent to
\eqref{eqn:kahler-def}, and strictly negative wherever the period map of
$F(s)$ has non-zero derivative.  See Example \eqref{exmple:kahler-def}.
The resulting \changed{}{pseudo-metric is also  K\"ahler}, cf. Corollary \eqref{corr:kahler}.

\item
Turning now to \emph{algebraic cycles}, recall that
by~\cite{ANF}, a normal function is equivalent to an
extension in the category of variations of mixed Hodge structure\footnote{Note:
We have performed a Tate twist to make $\mathcal H$ have weight -1
here.}
\begin{equation}
      0 \to\mathcal H \to \VV \to \Z(0) \to 0. \label{eqn:nf-2}
\end{equation}
The classical example  comes from the Abel-Jacobi
map for degree zero divisors on a compact Riemann surface
and its  natural extension
\begin{equation}
          \text{AJ}: \text{CH}^k_{\rm hom}(Y) \to J^k(Y)             \label{eqn:AJ-map}
\end{equation}
to homologically trivial algebraic cycles on a smooth complex projective
variety $Y$~\cite{periods}.  Application of this construction pointwise
to a family of algebraic cycles $Z_s\subset Y_s$ yields the prototypical
example of a \emph{normal function}
\begin{equation}
          \nu:S\to J(\mathcal H)                   \label{eqn:nf-1}
\end{equation}
where $\mathcal H$ is the variation of pure Hodge structure attached to
the family $Y_s$.

\begin{prop*}  1. The pullback of the mixed
Hodge metric along a normal function is a pseudo-K\"ahler (c.f. Example \ref{kaehlerexmples}).\\
2. In the case where the underlying variation of pure Hodge structure is
constant (e.g. a family of cycles on a fixed smooth projective variety $Y$),
the holomorphic sectional curvature is semi-negative
(Corollary~\ref{curvneg:2}).
\end{prop*}


\par Using the polarization of $\mathcal H$, one  can construct a natural
\emph{biextension} line bundle $B\to S$ whose fibers parametrize mixed Hodge
structures with graded quotients
$$
     \gr^W_0\cong\Z(0),\qquad \gr^W_{-1}\cong\mathcal H_s,\qquad
     \gr^W_{-2}\cong\Z(1)
$$
and such that the extension between $\gr^W_0$ and $\gr^W_{-1}$ is determined by
$\nu(s)$ and the extension from $\gr^W_{-1}$ and $\gr^W_{-2}$ is determined
by the dual of $\nu(s)$.

\par As noted by Richard Hain, the biextension line bundle $B$ carries a
natural hermitian metric $h$ which is based on measuring how far the mixed
Hodge structure defined by $b\in B_s$ is from being split over $\mathbb R$.
In~\cite{nilp}, the first author and T. Hayama prove that for
$B\to\Delta^{*r}$ arising from
an admissible normal function with unipotent monodromy, the resulting
biextension metric is of the form
\begin{equation}\label{eqn:bm}
      h = e^{-\varphi}
\end{equation}
with $\varphi\in L^1_{\rm loc}(\Delta^r)$, i.e. it defines a singular hermitian
metric in the sense of~\cite{dem} and hence can be used to compute the
Chern current of the extension of $\bar B$ obtained by declaring the
admissible variations of mixed Hodge structure to define the extending
sections (cf.~\cite{nilp,bp}). For this situation we show (\S \ref{sec:biexts}):
\begin{prop*} Let $S$ be a curve and let  $\mathcal B $ be a variation of
biextension type over $S$. Then the Chern form of the biextension metric
\eqref{eqn:bm}  is the $(1,1)$--form
\begin{equation*}
           -\frac{1}{2\pi\ii} \del\bar\del h (s)
      = \half [\gamma^{-1,0},\bar\gamma^{-1,0}] \,ds \wedge \overline{ds},
\end{equation*}
where $\gamma^{-1,0}$ is the Hodge component of type $(-1,0)$ of
$\varphi_*(d/ds)$ viewed as an element of $\mathfrak g_\C$.  For self-dual
variations this  form is semi-negative.
\end{prop*}

\begin{rmq} This result was also obtained Richard Hain (\S 13, \cite{hain-2})
by a different method.
\end{rmq}

We then deduce (see Cor.~\ref{biextcoroll} for a precise statement):

\begin{corr*} Let $\mathcal B$ be a self-dual biextension over $S$ with
associated normal function $\nu$.  Then, the Chern form of the biextension
metric vanishes along every curve in the zero locus of $\nu$.
\end{corr*}

\par
\changed{}{
\noindent As a final remark, we want to mention that  the
 asymptotic behavior of the biextension metric
is  related to} the Hodge conjecture:  Let $L$ be a very ample line
bundle on a smooth complex projective variety $X$ of dimension $2n$
and $\bar P$ be the space of hyperplane sections of $X$.  Then, over
the locus of smooth hyperplane sections $P\subset\bar P$, we have
a natural variation of pure Hodge structure $\mathcal H$ of weight
$2n-1$.  Starting from a primitive integral, non-torsion Hodge class
$\zeta$ of type $(n,n)$ on $X$, we can then construct an associated
normal function $\nu_{\zeta}$ by taking a lift of $\zeta$ to Deligne
cohomology. The Hodge conjecture is then equivalent~\cite{gg,bfnp}
to the existence of singularities of the normal function $\nu_{\zeta}$
(after passage to sufficiently ample $L$).  \changed{}{In~\cite{bp}, it is
  shown} that the existence of singularities of $\nu_{\zeta}$ is
detected by the failure of the biextension metric to have a
smooth extension to $\bar P$.

\end{enumerate}

%

\subsection{Structure}  We start properly in \S \ref{sec:class}  and summarize the basic
properties of the classifying spaces of graded-polarized mixed Hodge
structures following~\cite{higgs}
and compute the dependence of the
bigrading \eqref{eqn:DelSplit} on $F\in D$ up to second order.  Using
these results, we then compute the curvature tensor and the holomorphic sectional curvature of $D$ in \S \ref{sec:curv}--\ref{sec:holseccurv}.

\par

 \par In \S\ref{sec:hb} and  \S\ref{sec:biexts} we compute the curvature of the Hodge bundles and the
biextension metrics using similar techniques. Likewise, in \S \ref{sec:kaehler} we
use the computations of \S \ref{sec:holseccurv}  to determine when the pull back of the
mixed Hodge metric along a period map is K\"ahler.     In \S\ref{defoinAb} we show how these calculations apply to particular
situations of geometric interest.

\par In \S \ref{sec:reductive}, we construct a classifying space $D$ which is a reductive
domain such that its natural complex structure is not compatible with the
usual complex structure making the Hodge metric a hermitian equivariant
metric.   So  the Chern connection for the Hodge metric is not the same as
the one coming from the Maurer-Cartan form on $G_\C$.  This makes the
calculations in the mixed setting intrinsically more involved than in the
pure case, even in the case of a split mixed domain.

\par
In Appendix~\ref{Appendix A} we compute the Levi-Civita connection for the Hodge metric.
In general it does not conserve the splitting of the complex tangent bundle into the holomorphic and anti-holomorphic
parts which makes the formulas more complicated than the one for the Chern connection. Nevertheless  in certain cases it simplifies which has in favorable cases consequences for the curvature and for geodesics (Cor.~\ref{LC4}).

\medskip
\noindent\textbf{Acknowledgements.}  Clearly, we should  first and foremost
 thank A. Kaplan  for his  ideas  concerning mixed domains and their metrics.

Next, we want to thank   Ph.~Eyssidieux , P.~Griffiths, S.~Grushevsky, R.~Hain,
C.~Hertling,  J. M.~Landsberg and  C.~Robles for their interest and pertinent
remarks.

The cooperation resulting in this paper started during a visit of the first
author to  the University of Grenoble; he expresses his thanks for its
hospitality.  The work of the first author was also partially supported
by NSF grant DMS 1361120.

\section{Classifying Spaces} \label{sec:class}

\subsection{Homogeneous Structure}\label{ssec:homstr}
We begin this section by reviewing some material on classifying
spaces of graded-polarized mixed Hodge structure \cite{usui} which appears in
\cite{higgs,dmj,sl2anddeg}.  Namely, in analogy with the pure case,
given a graded-polarized mixed Hodge structure $(F,W)$ with underlying real
vector space $V_\R$, the associated classifying space $D$ consists of all
decreasing filtrations of $V_{\C}$ which pair with $W$ to define a
graded-polarized mixed Hodge structure with the same graded Hodge numbers
as $(F,W)$.  The data {for} $D$ is therefore
$$
      (V_{\R},W_{\bullet},\{Q_{\bullet}\},h^{\bullet,\bullet})
$$
where $W_{\bullet}$ is the weight
filtration, $\{Q_{\bullet}\}$ are the graded-polarizations and
$h^{\bullet,\bullet}$ are the graded Hodge numbers.

\par To continue, we recall that given a point $F\in D$ the associated
bigrading \eqref{eqn:DelSplit} gives a functorial isomorphism
$
      V_{\C}\cong \gr^W
$
which sends $I^{p,q}$ to $H^{p,q}\subseteq \gr^W_{p+q}$ via the quotient map.
The pullback of the standard Hodge metrics on $\gr^W$ via this isomorphism
then defines a mixed Hodge metric on $V_{\C}$ which makes the bigrading
\eqref{eqn:DelSplit} orthogonal and satisfies
$$
       h_F(u,v) = \ii^{p-q}Q_{p+q}([u],[\bar v])
$$
if $u$, $v\in I^{p,q}$.  By functoriality, the point $F\in D$ induces a
mixed Hodge structure on $\text{End}(V)$ with bigrading
\begin{equation}
        \text{End}(V_{\mathbb C})
         = \bigoplus_{r,s}\, \text{End}(V)^{r,s}
         \label{eqn:induced-bigrading}
\end{equation}
which is orthogonal with respect the associated metric
\begin{equation}
         h_F(\alpha,\beta) = \tr(\alpha \beta^*)   \label{eqn:mhm}
\end{equation}
where $\beta^*$ is the adjoint of $\beta$ with respect to $h$.

\par Let $\text{GL}(V_{\C})^W\subset\text{GL}(V_{\C})$ denote
the Lie group of complex linear automorphisms of $V_{\C}$ which preserve the
weight filtration $W$.  For $g\in\text{GL}(V_{\C})^W$ we let
$Gr(g)$ denote the induced linear map on $\gr^W$.  Let $G_{\mathbb C}$ be
the subgroup consisting of elements which induce complex automorphisms of the
graded-polarizations of $W$, and $G_{\R} = G_{\mathbb C}\cap GL(V_{\mathbb R})$.

\par In the pure case, $G_{\R}$ acts transitively on the classifying space
and $G_{\C}$ acts transitively on the compact dual.  The mixed case is
slightly more intricate:  Let $G$ denote the subgroup of elements of $G_{\C}$
which act by real transformations on $\gr^W$.   Then,
$$
       G_{\R}\subset G \subset G_{\C}
$$
and we have the following result:

\begin{thm}[\protect{\cite[\S 3]{higgs}}]  The classifying space  $D$ is a
complex manifold upon which $G$ acts transitively by biholomorphisms.
\end{thm}

\begin{rmq}   Hertling~\cite{pmhs} defines a period domain   of polarized mixed Hodge structures
on a fixed real vector space $V$ equipped with a polarization $Q$ and  weight filtration induced by a nilpotent  infinitesimal isometry
$N$ of $(V,Q)$.  The difference with our approach  is that the latter domain is homogeneous under the subgroup of $G$ consisting of elements commuting with $N$. So  in a natural way it is a submanifold of our domain.
\end{rmq}

\subsection{Hodge Metric on the Lie Algebra} \label{ssec:tang}
Let $\eg_{\R} = \text{Lie}(G_{\R})$ and $\eg_{\C} = \text{Lie}(G_{\C})$.
By functoriality, any point $F\in D$ induces a mixed Hodge structure
on $\eg_{\C}=\eg_{\R}\otimes\C$ with bigrading inherited from the one on
$\End(V_\C)$, i.e. $\eg ^{r,s} = \eg_\C\cap \End(V)^{r,s}$. For future
reference, we note that:
\begin{itemize}
\item $\eg_\C\cap \End(V)^{r,s}=0$ if $r+s>0$;
\item $W_{-1}\End(V) \subset \eg_{\C}$.
\item The orthogonal decomposition
\begin{equation}
     \End(V_{\C}) = \eg_{\C}\oplus\eg_{\C}^{\perp}      \label{eqn:orth-1}
\end{equation}
induces a decomposition
\begin{equation}
       \End(V)^{p,-p} = \eg_{\C}^{p,-p}\oplus (\eg_{\C}^{\perp})^{p,-p}
       \label{eqn:orth-2}
\end{equation}
\item Let $*$ denote ajoint with respect to the metric $h_F$.  Then,
\begin{equation}
           *:\End(V)^{p,q}\to \End(V)^{-p,-q} ;     \label{eqn:adjoint-pq}
\end{equation}
\item By Lemma \eqref{lemma:adjoint} below $\alpha\in\eg^{p,-p}\implies\alpha^*\in\eg^{-p,p}$.
\end{itemize}


\begin{rmk} In general, for a mixed Hodge structure which is not split over
$\mathbb R$, the operations of taking adjoint with respect to the mixed
Hodge metric and complex conjugate do not commute.
\end{rmk}

\par Let $\text{Flag}(D)$ denote the flag variety containing $D$, i.e.
the set of all complex flags of $V_{\C}$ with the same rank sequence
as the flags parametrized by $D$.  Then, since $G\subset G_{\C}$ acts
transitively on $D$, it follows that the orbit of any point $F\in D$
under $G_{\C}$ gives a well defined \lq\lq compact dual\rq\rq{}
$\check D\subset\text{Flag}(D)$ upon which $G_{\C}$ acts transitively
by homeomorphisms:
\begin{equation}
\check D = G_\C/G_\C^F.\label{eqn:CompactDual}
\end{equation}

\begin{rmk} As in the pure case, $D$ is an open subset of
$\check D$ with respect to the analytic topology.  In the mixed case
however, $\check D$ is usually not compact: in Example~\ref{ssec:exmples}.3
one has $G=G_\C$ and hence $D=\check D= \C\subset \text{Flag}(D)=\bP^1$.
One could consider the
closure of $\check D$ in the ambient flag variety to obtain a
compact object, but  as the example shows, this need not be a homogeneous space for
$G_{\C}$.
\end{rmk}

\begin{rmk} \label{mMTD} In analogy with the above, one defines the \emph{mixed Mumford--Tate
domains} as follows:  Let $(F,W)$ be a graded-polarized mixed Hodge structure
with MT group $M$ and $M_{\text { split}}$ be the direct sum of the Mumford--Tate
groups of the associated pure Hodge structures on $Gr^W$.  Then, $M$ is an
extension of $M_{\text{split}}$ by a unipotent group $U$.   Let $\mathfrak u$ denote
the Lie algebra of $U(\mathbb C)$ viewed as a real Lie algebra and
$\mathfrak m_{\mathbb R}$ denote the Lie algebra of $M(\mathbb R)$.  Let $G_M$
denote the real Lie group with Lie algebra $\mathfrak u + \mathfrak m_{\mathbb R}$
viewed as a real subalgebra of $\Lie{M({\C})}$.  Then, the associated
mixed Mumford--Tate domain $D_M$ is the orbit of $F$ under $G_M$. 

The proof that $D_M$ is a complex manifold is parallel to the
proof for $D$: The compact dual $\check D_M$ is the complex homogeneous
space defined by the orbit of $F$ under $M(\mathbb C)$, and hence it is
sufficient to check that there exists a neighborhood $O$ of $1\in M(\mathbb C)$
such that $O\cdot F\subset D_M$.

It follows that in subsequent calculations we may replace  $\eg_\C$ by $\Lie{M(\C)}$.
\end{rmk}

By the defining properties of the bigrading~\eqref{eqn:DelSplit},
it follows that
\begin{equation}
     \eg_{\C}^F = \bigoplus_{r\geq 0}\, \eg^{r,s}    \label{eqn:isotopy-alg}
\end{equation}
is the Lie algebra of the stabilizer of $F\in D$ with respect to the
action of $G_{\C}$ on $\check D$.  Accordingly,
\begin{equation}
      \eq_F = \bigoplus_{r<0}\, \eg^{r,s}       \label{eqn:tang-alg}
\end{equation}
is a vector space complement to $\eg_{\C}^F$ in $\eq_{\C}$ and hence:

\begin{lemma}\label{TangentIdent}  The map
$$
        u\in \eg_{\C} \mapsto \gamma_*(d/dt)_0,\qquad \gamma(t) = e^{tu}\cdot F
$$
determines an isomorphism between $\eq_F$ and $T^{\rm{hol}}_F(D)$.
\end{lemma}

The preceding Lemma gives a way to induce a hermitian metric on the tangent bundle $T(D)$:
\begin{dfn} \label{HMonD} The isomorphism \eqref{eqn:tang-alg} provides  $D$ with a  metric, the \emph{Hodge metric}.
\end{dfn}

\par For $F\in D$ let $\pi_{\eq}$ denote orthogonal projection $\End(V_{\C})\to\eg_{\C}$.  We note
that the restriction of $\pi_{\eq}$ to $\eg_{\C}$ is just projection with respect to the decompostion
\begin{equation}
            \eg_{\C} = \eg_{\C}^F\oplus\eq_F .         \label{eqn:gc-decomp}
\end{equation}

\begin{lemma} \label{compatibilities} Let $f\in \eg_\C^F$. Then,
\begin{equation}
\pi_{\eq}\comp (\ad{f})^n
    = \left(\pi_{\eq}\comp  \ad f\right)^n \label{eqn:condense-1}
\end{equation}
as linear operators on $\eg_\C$.
\end{lemma}
\proof
Induct on $n$, with the base case $n=1$ a tautology. Observe that
\begin{equation}
         (\ad{f})^{n} u = v + w .              \label{eqn:induct-1}
\end{equation}
with $v\in\eq_F$ and $w\in \eg^F_{\C}$.  Therefore,
$
       (\ad{f})^{n+1} u = [f,v] + [f,w]
$
and hence
\begin{equation}
        \pi_{\eq}((\ad{f})^{n+1} u) = \pi_{\eq}[f,v].  \label{eqn:induct-2}
\end{equation}
By equation \eqref{eqn:induct-1}, $v=\pi_{\eq}((\ad{f})^{n} u)$ which is
equal to $(\pi_{\eq}\comp \ad{f})^n u$ by induction.  Substituting this
identity into \eqref{eqn:induct-2} gives
$$
       \pi_{\eq}((\ad{f})^{n+1} u) = (\pi_{\eq}\comp \ad{f})^{n+1} u. \hfill\qed
$$

\endproof


\par Before stating the next result, we emphasize that unlike the pure case, the operation of taking
adjoint with respect to the mixed Hodge metric does not preserve $\eg_{\C}$.  Therefore, the statement
and proof of the next result all occur in the Lie algebra $\End(V)$.

\begin{corr} Let $f\in \eg_\C^F$ and $v$, $w\in\eq$.  Then,
\begin{equation}
   h_F(v,\exp(\pi_{\eq}\comp \ad f)w) = h_F(\exp(\pi_{\eq} \comp \ad f^*)v,w)
   \label{eqn:condense-2}
\end{equation}
\end{corr}
\begin{proof} It is sufficient to prove
$$
        h_F(v, (\pi_{\eq}\comp \ad f)^m\, w) = h_F((\pi_{\eq} \comp \ad f^*)^m\, v, w)
$$
We induct on $m$.  For $m=1$ we have
$$
       h_F(v,\pi_{\eq}[f,w]) = h_F(v,[f,w]) = h_F([f^*,v],w) = h_F(\pi_{\eq}[f^*,v],w)
$$
since $[f,w] = w' + w''$ with $w'\in\eq$ and $w''\in \eq^{\perp}$, which justifies
$$
       h_F(v,\pi_{\eq}[f,w]) = h_F(v,[f,w]) = h_F([f^*,v],w)
$$
Likewise, $[f^*,v] = v' + v''$ with $v'\in\eq$ and $v''\in\eq^{\perp}$ and so
$$
         h_F([f^*,v],w) = h_F(\pi_{\eq}[f^*,v],w)
$$
Since at each stage we project onto $\eq$, passage from $m$ to $m+1$ follows from the
formula for $m=1$.
\end{proof}

\par Define
\begin{equation}\label{eqn:lambda}
 \Lambda = \bigoplus_{r,s<0}\, \eg^{r,s}
\end{equation}
and note that since the conjugation condition appearing in
\eqref{eqn:DelSplit} can be recast as
\begin{equation} \label{eqn:UnderConju}
\bar \eg^{p,q}\subset \eg^{q,p}+[\Lambda, \eg^{q,p}],
\end{equation}
it follows that $\Lambda$ has a real form
\begin{equation}\label{eqn:LambdaRealForm}
\Lambda_{\R} = \Lambda\cap\eg_{\R}.
\end{equation}
\begin{lemma}[\protect{\cite[Lemma 4.11]{higgs}}] \label{higgs} If
$g\in G_\R\cup\exp(\Lambda)$ then
$$
      g (I_F^{p,q})= I_{g\cdot F}^{p,q}.
$$
\end{lemma}
%
Recall that a mixed Hodge structure $(F,W)$ is said to be \emph{split
over $\R$} if
$$
      \overline{I^{p,q}} = I^{q,p}.
$$
Those mixed Hodge structures make up a real  analytic subvariety $D_\R\subset D$.
To any given mixed Hodge structure $(F,W)$, one associates a  special split real mixed Hodge structure  $\hat F=e_F\cdot F$
as follows.
\begin{prop}[\protect{\cite[Prop. 2.20]{degeneration}}] \label{SplitMHS} Given
a mixed Hodge structure   there is a unique $\delta \in \Lambda_\R $ such that
the spaces $\hat I^{p,q} =\exp{(-\ii \delta)}I^{p,q}$ give  the   splitting of a
split real mixed Hodge structure $\hat F= e_F\cdot F$, the \emph{Deligne splitting}.
\end{prop}
\par A \emph{splitting operation}  is  a particular type of fibration $D\to D_\R$
of $D$ over the locus of split mixed Hodge structures (cf.
\cite[Theorem (2.15)]{sl2anddeg}).   Our calculations below use the following result due to Kaplan:

\begin{thm}[\protect{\cite{kaplan}}]\label{group-decomp} Given a choice of
splitting operation and choice of base point $F\in D$, for each element
$g\in G$ exists a distinguished decomposition
\[
    g = g_{\R}\exp{(\lambda)} f ,\qquad
    \lambda\in   \Lambda ,\quad
    g_{\R}\in G_{\R},
    \quad f\in \exp( W_{-1} \eg\el(V_{\C}))\cap G^F.
\]
Moreover, if the splitting operation is an analytic or $C^\infty$ map, the map
$(F,g)\mapsto (g_\R,e^\lambda, f)$ is analytic, respectively $C^\infty$.
\end{thm}


\par Using the  identification of $T_FD$ with $\eq_F$ as given by Lemma~\ref{TangentIdent}, the mixed Hodge metric
\eqref{eqn:mhm} induces a hermitian structure on $D$.
In analogy with Lemma \eqref{higgs} and the fact that $G$ acts
by isometry on $\gr^W$ it follows that

\begin{lemma}[\cite{kaplan,sl2anddeg}] \label{KaplanChange} For any
$g=g_\R e^\lambda$, $g_\R\in G_\R$, $\lambda\in \Lambda$, the mixed Hodge
metric on $\eg_\C$ changes equivariantly:
\[
  h_{g\cdot F }( \Ad{g}\alpha ,\Ad {g}\beta )= h_F(\alpha ,\beta ) ,
    \,\forall \alpha ,\beta \in \eg.
\]
and hence $g:T_F(D)\to T_{g\cdot F}(D)$ is an isometry.
\end{lemma}

\begin{rmk} (1) In~\cite{knu,knu2}, the authors consider a different metric
on $D$ which is obtained by replacing the bigrading \eqref{eqn:DelSplit}
attached to $(F,W)$ by the bigrading attached to the \emph{canonical} or
\emph{ $\slgr 2$-splitting} of $(F,W)$.  They then twist this metric by a
distance to the boundary function \cite[\S 4]{knu2}.  
The metric of~\cite{knu,knu2} is not quasi-isometric to the metric
considered in this paper except when $D$ is pure.  See \cite{nilp} for
details on the geometry of this metric.\\
(2)   The previous Lemma implies that, understanding how the decomposition
appearing in Theorem~\ref{group-decomp} depends on $F\in D$ up to second order
is sufficient to compute the curvature of $D$ (cf.\cite{deligne}).
\end{rmk}

\par  For future
use, we introduce the subalgebras
\begin{equation}
     \eun_+  :=  \bigoplus_{ a\ge 0,b<0  } \eg^{a,b},\qquad
     \eun_-  :=  \bigoplus_{ a< 0,b\ge 0  } \eg^{a,b}.   \label{eqn:SplitEnd}
\end{equation}
 Then, recalling the definition \eqref{eqn:lambda} of $\Lambda$, we have a splitting
$$
      \eg_{\C} = \eun_+\oplus\eg^{0,0}\oplus\eun_-\oplus\Lambda
$$
and we let
\begin{equation}
\label{eqn:projections}
\aligned
\End(V_\C) & \to & \eun_+,\, \eg^{0,0},  \eun_- ,\, \Lambda\\
u  & \mapsto &  \, u_+, \,u_0,\, u_-, \,u_\Lambda
\endaligned
\end{equation}
denote orthogonal projection from $\End(V_{\C})$ to $\eg_{\C}$ followed by projection onto the
corresponding factor above.

\par 
We conclude this section with a formula for the adjoint operator
$\alpha\mapsto\alpha^*$ with respect to the mixed Hodge metric.
\begin{lemma}\label{lemma:adjoint}
Let $\ez= \bigoplus_p \eg^{-p,p}$ and denote
\begin{equation}\label{eqn:piz}
\pi_\ez: \End(V_\C) \to \ez
\end{equation}
the corresponding orthogonal projection. Then (with $C_F$ the Weil operator of $\gr^WV$) we have
$$
       \alpha \in\ez \implies\alpha ^*=-\Ad{(C_F)} \pi_\ez(\bar \alpha ).
$$
\end{lemma}
\proof  In the pure case, the statement is well known. Since both sides belong
to $\ez$, we only have to check that we get the correct formula on
$\gr^W_0(\eg_{\mathbb C})$.
\qed\endproof

\subsection{Second Order Calculations}  \label{ssec:secondorder} In this
subsection, we compute the second order behavior of the decomposition of
$g = \exp(u)$ given in Theorem~\ref{group-decomp}.  The analogous results
to first order appear in \cite{higgs}.

\par Employing the notation\footnote{We simplify notation by writing $\eq$
instead of $\eq_F$.} from \eqref{eqn:tang-alg}  and \eqref{eqn:SplitEnd}
consider the following splitting
\begin{equation}\label{eqn:Double}
\eg_\C=\underbrace{\eg^{0,0}\oplus \eun_+}_{\eg_\C^F} \oplus
       \underbrace{\eun_-\oplus\Lambda}_{\eq}.
\end{equation}
Since $\eq$ is a complement to $\eg_\C^F$, the map
\begin{equation}
      u\in\eq \mapsto e^u\cdot F     \label{eqn:coord}
\end{equation}
restricts to biholomorphism of a neighborhood $U$ of $0$ in $\eq$ onto
a neighborhood of $F$ in $D$.  Relative to this choice of coordinates,
the identification of $\eq$ with $T_F(D)$ coincides with the one
considered above (cf. \eqref{eqn:tang-alg}).

\par We need to compare this with the real structure on
$\eg_\C = \eg_{\R}\otimes\C$.  As usual, we write
\[
 \alpha= \re (\alpha) + \ii \cdot  \im (\alpha) ,\quad
 \re (\alpha) =\half (\alpha+\bar\alpha),\quad
 \ii \cdot  \im (\alpha)= \half( \alpha- \bar \alpha).
\]

  \begin{lemma*} [\protect{\cite[Theorem 4.6]{higgs}}]
Set
\[
\Im  (\eg^{0,0}):= \sett{\varphi\in\eg^{0,0}}{\bar\varphi^{(0,0)}=-\varphi} .
\]
Then
\begin{equation}\label{eqn:decomposition}
\eg_\C= \eg_\R\oplus \Im  (\eg^{0,0})  \oplus  \eun_+\oplus  \ii \Lambda_\R.
\end{equation}
\end{lemma*}



\begin{corr}[\protect{\cite[Corollary 4.7]{higgs}}] There exists a neighborhood of
$1\in G_\C$ such that every element $g$ in this neighborhood can be written
uniquely as
\[
    g=g_\R \exp{( \lambda)} \exp(\varphi), \quad
    g_\R\in\eg_\R,\, \lambda\in \ii \Lambda_\R, \quad
    \varphi\in    \eg^{0,0}      \oplus  \eun_+ \subset \eg^F_\C,
\]
where $\varphi^{0,0}$ is purely imaginary.
\end{corr}
This implies that, possibly after shrinking $U$  there are unique functions
$\gamma,\lambda, \varphi: U \to \eg_\R, \ii \Lambda_\R, \eg_\C^F$ respectively
such that
\begin{equation}\label{eqn:expu}
\exp(u) = \underbrace{\exp{(\gamma(u) )}}_{\text{\rm in }G_\R}\cdot
          \exp{( \lambda(u))} \cdot
          \underbrace{\exp{(\varphi(u))}}_{\text{\rm in  }G^F_\C}.
\end{equation}
Now we  introduce  $g(u)=\exp(u) =g_\R(u)\cdot \exp(\lambda(u) ) \cdot
\exp{(\varphi(u))}$ as functions on $U\cap \eq$.
\medskip

\par As a prelude to the next result, we recall that by the
Campbell--Baker--Hausdorff formula we have
$$
     e^x e^y = e^{x+y+\frac{1}{2}[x,y] + \cdots}.
$$
Alternatively, making the change of variables $u=-y$, $v=x+y$ this can
be written as
$$
       e^{u+v}e^{-u} =  e^{\psi(t_0,t_1,\dots)},
$$
where $t_m = (\ad u)^mv$ and $\psi$ is a universal Lie polynomial.
In a later computation (see the proof of Lemma~\ref{pushforwardtangents}) we need more information,
namely on the shape of the part linear in  $v$:
\begin{equation}
        \psi_1(u,v) = \sum_m\, \frac{1}{(m+1)!}\,t_m = \frac{e^{\ad u} - 1}{\ad u}v. 
                                                          \label{eqn:2.10}
\end{equation}

\begin{prop}\label{secondorder} Let $F\in D$ and
$u= u_-+u_\Lambda \in  \eun_-\oplus\Lambda = \eq$.
Then,
\[
     \varphi(u)= -\bar u_+ + \half[u, \bar u]_0 + [u, \bar u   ]_+
     + \half[\bar u,    { \bar{u}_\Lambda}]_+ + O^3(u,\bar u)
\]
where the subscripts mean the orthogonal projections onto $\eg^{0,0}$,
$\Lambda$, $\eun_+$ respectively.
\end{prop}
\proof
For the linear approximation note that
\[
     u= \re[2  (u_-) - \bar{u}_\Lambda ]
     - \ii \im  (\bar{u}_\Lambda) -\bar u_+  \in \eg_\R\oplus \ii\Lambda_\R\oplus \eg^F_\C
\]
and that equation \eqref{eqn:expu} yields the first degree approximation
$u=\gamma_1(u)+\lambda_1(u)+\varphi_1(u)$ so that the result follows by
uniqueness.

\par The computation proceeds by expanding the left hand side of
$$
     \exp{(\lambda)}\exp{(\varphi)} \exp{(-u)} =\exp{(-\gamma)} \in G_\R
$$
using the Campbell--Baker--Hausdorff formula, and then using the fact that
the right hand side is real.  To first order the decomposition is
$$
      u = \gamma_1(u) + \lambda_1(u) + \varphi_1(u)
$$
where
$$
\aligned
     \gamma_1(u)  &= u + \bar u - \frac{1}{2}\pi_{\Lambda}(\bar u)
                              - \frac{1}{2}\overline{\pi_{\Lambda}(\bar u)} \\
     \lambda_1(u) &= -\half\pi_{\Lambda}(\bar u)
                     + \frac{1}{2}\overline{\pi_{\Lambda}(\bar u)}          \\
     \varphi_1(u) &= -\bar u + \pi_{\Lambda}(\bar u)
\endaligned
$$
where we have used $\pi_{\Lambda}$ to denote projection to $\Lambda$ for
clarity regarding the order of complex conjugation, since these two
operations do not commute.

The second degree approximation then yields that
\[
     \lambda_2 +  \varphi_2
      + \half\left(  [\lambda_1,\varphi_1-u] -[\varphi_1,u]   \right)
      \text { is real.}
\]
The projection to $\eun_+$ equals $[\varphi_2]_+
+ \half\left(  [\lambda_1, \varphi_1 -u] _+ -[ \varphi_1,u ] _+\right)$.
Since $\bar\lambda_1=-\lambda_1$, the reality constraint implies that
\[
     \aligned
        (\varphi_2)_+ &=
         - \half  \left\{   [\lambda_1,\varphi_1+\bar\varphi_1-u-\bar u]_+
                            +[\bar\varphi_1,\bar u]_+-[\varphi_1,u]_+ \right \}\\
                            &=
                - \half  \left\{ [\bar\varphi_1,\bar u]_+-[\varphi_1,u]_+ +[\lambda_1,\varphi_1+\bar\varphi_1-u-\bar u]_+  \right\}            .
\endaligned
\]
By the conjugation rules $\bar \eun^\pm \subset \bar \eun_\mp +\Lambda$, the fact that $\Lambda$, $\eun_+,\eun_-$ are subalgebras,  and using $[\eun_{\pm},\Lambda] \subset \eun_{\pm}+\Lambda$ this simplifies to
\[
(\varphi_2)_+  = -\half \left\{ [\lambda_1,\varphi_1-\bar u]_+ + [\bar\varphi_1,\bar u]_+-  [\varphi_1,  u]_+ \right\} .
\]
Now set  $\varphi_1= -\bar u + \pi_\Lambda(\bar u)$  so that $\varphi_1-\bar u= -2 \bar u \mod \Lambda$.   The first term thus reads
$\half [2\lambda_1, \bar u]_+$, and since  $ \overline{\varphi_1} = -\overline{\pi_+\bar u }$,
 the second term becomes $\half[\overline{\bar u_+}, \bar u]_+$ while the last simplifies to $-\half[\bar u, u]_+$ ; in total we get
 \[
 (\varphi_2)_+  = \half [2\lambda_1+  \overline{\pi_+\bar u } , \bar u]_+ +\half[\lambda_1,\bar u]_+.
 \]
Putting
$ 2\lambda_1  =   \overline{ \pi_\Lambda \bar u} -  \pi_\Lambda(\bar u)$ so that $2\lambda_1+\overline{\pi_+\bar u}= u-   \pi_\Lambda(\bar u)$
  shows
\[
(\varphi_2)_+= \half\left\{  [u,\bar u]_+ -[\bar u_\Lambda,\bar u]_+ - [\bar u,u]_+\right\},
\]
which is indeed equal to the stated expression for $(\varphi_2)_+$.
Similarly we find for the $\eg^{0,0}$-component
\[
       (\varphi_2)_0= \half [u,\bar  u]_0.\quad\quad \qed
\]
 \endproof

\begin{corr}\label{corr:H-form}  Let $F\in D$.  Let
$$
      h_{e^u\cdot F}(L_{e^u *}\alpha ,L_{e^u *}\beta )
       = h_F(\exp{H(u)}\alpha ,\beta ),\quad \alpha,\beta\in\eq
$$
denote the local form of the mixed Hodge metric on $T(D)$ relative to the
choice of {coordinates} \eqref{eqn:coord}.
Then, up to second order in\footnote{We write $x_\eq$ instead of $\pi_\eq x$ for clarity  and if no confusion is likely.}
$(u,\bar u)$
\begin{eqnarray*}
H(u) &=&
\underbrace{- (\ad{(\bar u)_+^*)}_\eq}_{(1,0)\text {-\rm term} }  +\underbrace{ - (\ad{ (\bar u)_+})_\eq}_{(0,1)\text {-\rm term}}\\
&& +  \underbrace{ \frac 12   \left(
\ad{ [\bar u ,    \bar  u_\Lambda  ]_+
       + [\bar u ,   \bar u_\Lambda ]_+^*} \right)_\eq}_{(2,0)+(0,2) \text{-\rm term}}  \\
&& + \underbrace{
 \left( \half [ (\ad{ (\bar u)_+^*})_\eq  ,(\ad{(\bar u)_+})_\eq ]  +(\ad{     [ u, \bar u]_0})_\eq %
  + \ad{[u,\bar u]_+}  +\ad {[u,\bar u]_+^*} \right)_\eq }_{(1,1) \text{-\rm term}}.
  \end{eqnarray*}
  Here, by "$A(x,y)$ is a $(p,q)$-term" we mean $A(tx,ty)= t^p\bar t^q A(x,y)$.
\end{corr}
\proof  Let us first check the assertion about types. This follows directly from the the facts that $\ad{}$  and $\pi_\eq$ are $\C$-linear, while for any $\C$-linear operator $A$, one has $(tA)^*=\bar t A^*$ and $\overline{tA}=\bar t \bar A$.

Let us now start the calculations. By \eqref{eqn:expu}, we have
\begin{equation}\label{eqn:intermediatestep}
\aligned
      h_{e^u\cdot F}(L_{e^u *}\alpha ,L_{e^u *}\beta )
      &= h_F(L_{\exp(\varphi(u)*}\alpha ,L_{\exp(\varphi(u)*}\beta ))         \\
      &= h_F(\pi_{\eq}\Ad{\exp(\varphi(u))}\alpha ,
                      \pi_{\eq}\Ad{\exp(\varphi(u))}\beta )) \\
      &= h_F(\pi_{\eq}\Ad{\exp(\varphi(u))}\alpha ,
                      \Ad{\exp(\varphi(u))}\beta ))
\endaligned
\end{equation}
since $\eg_{\C}^F$ and $\eq$ are orthogonal with respect to the mixed
Hodge metric at $F$.  Therefore,
$$
\aligned
      h_{e^u\cdot F}(L_{e^u *}\alpha ,L_{e^u *}\beta )
      &= h_F( \Ad{\exp(\varphi(u))^*}
              \pi_{\eq}\Ad{\exp(\varphi(u))}\alpha ,\beta ))  \\
      &= h_F( \exp(\ad{\varphi(u)^*})
              \pi_{\eq}\exp(\ad{\varphi(u)})\alpha ,\beta ))  \\
      &= h_F( \exp(\ad{\varphi(u)^*})
               \exp(\pi_{\eq}\ad{\varphi(u)})\alpha ,\beta ))  \\
\endaligned
$$
by equation \eqref{eqn:condense-1}.  Likewise, although
$$
     \exp(\ad{\varphi(u)^*})
               \exp(\pi_{\eq}\ad{\varphi(u)})\alpha
$$
is in general only an element of  $\End(V_{\C})$, since we are pairing
it against an element $\beta \in\eq$, it follows that
$$
\aligned
      h_{e^u\cdot F}(L_{e^u *}\alpha ,L_{e^u *}\beta )
      &=  h_F(\pi_{\eq}\exp(\ad{\varphi(u)^*})
               \exp(\pi_{\eq}\ad{\varphi(u)})\alpha ,\beta ))        \\
      &=  h_F(\exp(\pi_{\eq}\ad{\varphi(u)^*})
               \exp(\pi_{\eq}\ad{\varphi(u)})\alpha ,\beta )),
\endaligned
$$
where the last equality  follows from \eqref{eqn:condense-2}.
By the Baker--Campbell--Hausdorff formula, up to third order in $(u,\bar u)$ the product of the exponents in the previous formula can be replaced by
\[
 \exp\left(\pi_{\eq}\ad{\varphi(u)^*} + \pi_{\eq}\ad{\varphi(u)} + \frac{1}{2}[\pi_{\eq}\ad{\varphi(u)}^*,
                                 \pi_{\eq}\ad{\varphi(u)}]\right).
\]
So, we may assume that
\[
H(u)= \pi_{\eq}\ad{\varphi(u)^*} + \pi_{\eq}\ad{\varphi(u)} + \frac{1}{2}[\pi_{\eq}\ad{\varphi(u)}^*,
                                 \pi_{\eq}\ad{\varphi(u)}].
\]
To obtain the stated formula for $H(u)$, insert  the formulas from Proposition~\ref{secondorder}
into the above equations and compute up to order 2 in $u$ and $\bar u$.   Use is made of the equality $[u,\bar u]_0^*= [u,\bar u]_0$  guaranteed by Lemma~\ref{lemma:adjoint}.
\qed\endproof


\section{Curvature of the Chern Connection} \label{sec:curv}

\par We begin this section by recalling that given a holomorphic vector
bundle $E$ equipped with a hermitian metric $h$, there exists a unique
{\it Chern connection} $\nabla$ on $E$ which is compatible with both $h$ and
the complex structure $\bar\partial$.  With respect to any local holomorphic
framing of $E$, the connection form of $\nabla$ is given by
\begin{equation}
    \theta = \text{\bf h}^{-1}\partial\text{\bf h} ,  \label{eqn:conn-1}
\end{equation}
where $\text{\bf h}$ is the transpose of the Gram--matrix of $h$ with
respect to the given frame.  The curvature tensor is then
\begin{equation}
    R_{\nabla} = \bar\partial\,\theta.                         \label{eqn:conn-2}
\end{equation}

\begin{thm}\label{thm:connection} The connection one-form of the mixed
Hodge metric with respect to the trivialization of the tangent bundle
given in 
Lemma~\ref{TangentIdent}
is
$$
           \theta(\alpha) = -  
           \left(  \ad(\bar \alpha)_+^*\right)_\eq
$$
for $\alpha\in\eq\cong T_F(D)$.
\end{thm}
\proof By Corollary \eqref{corr:H-form}, this is the first order
holomorphic term of $H(u)$.\qed
\endproof

\begin{lemma}\label{lemma:curvature} Let $(D,h)$ be a complex hermitian
manifold and let $U\subset D$ be a coordinate neighborhood centered at
$F\in D$ and let $\alpha,\beta  \in T_F (U) \otimes\C$ be of type  $(1,0)$.  In a
local holomorphic frame, write the transpose Gram-matrix
$\mathbf{h}_U=(h(e^j,e^i) )= \exp H$ for some function $H$ with  with values
in the hermitian matrices and with $H(0)=0$.  Then at the origin one has
\[
    R_\nabla (\alpha ,\bar \beta )
     =   -\del_\alpha  \del_{\bar \beta  }  H
         + \frac 12  \left[  \del_{\bar\beta   }   H,  \del_{\alpha}     H \right].
\]
\end{lemma}
\proof Since the curvature is a tensor, its value on vector fields at a
given point only depends on the fields at that point.  Choose a complex
surface $u:V\hookrightarrow U$, $V\subset\C^2$ a neighborhood of $0$ (with coordinates
$(z,w)$) and
$u_*(d/dz)_0 = (\partial_\alpha  )_0$, $u_*(d/dw)_0 = (\partial_\beta  )_0$.  Replace $\text{\bf h}$ by
$\text{\bf h}\comp  u$ and write it as
$$
       h = \exp(H) = I + H + \frac{1}{2}H^2
                      + O^3(z,\bar z).
$$
Formulas \eqref{eqn:conn-1},\eqref{eqn:conn-2} tell us that the curvature at the origin equals
$$(\bar\del  h \wedge \del h +   \del \bar\del h)_0.$$
This $2$-form evaluates on the pair of tangent vectors $(\del_z,\del_{\bar w})$ as
\begin{equation}
\label{eqn:curve3}
R_\nabla(\alpha,\bar \beta)=  \del_{\bar w}h \comp \del_z h \ -\del_{z}\del_{\bar w}h.
\end{equation}
Now use the Taylor expansion of $h$ up to order 2 of which we give some  relevant terms\footnote{Remember $H$ is a matrix so that $\del_zH$ and $\del_{\bar w}H$ do not necessarily commute.}:
$$
\aligned
      h(z,\bar z,w,\bar w)_2 &= I + (\partial_z H)_0 z
                         +  (\partial_{\bar w} H)_0 {\bar w} + \text{ linear terms involving } \bar z, w  \\
                    &+\text{terms involving } z^2,w^2,\bar z^2,\bar w^2+ \\
                    &+\left(\partial_z\partial_{\bar w}\, H
                             +\frac{1}{2}(\partial_z\, H)(\partial_{\bar w}\, H)
                      +\frac{1}{2}(\partial_{\bar w}\, H)(\partial_z\, H)
                      \right)_0\, z\bar w  \\
                      & +\text{terms involving } z\bar z,w\bar z, w\bar w.
\endaligned
$$
Now substitute in  \eqref{eqn:curve3}.\qed\endproof

\par As a first consequence, we have:

\begin{lemma}\label{lemma:tate-case} The submanifold $\exp(\Lambda) \cdot F$
of $D$ is a flat submanifold with respect to the Hodge metric.  In particular,
the holomorphic sectional curvature in directions tangent to this submanifold
is identically zero.
\end{lemma}
\proof If $\mathbf{f}$ is a unitary  Hodge-frame for the mixed Hodge
structure on $V$  corresponding to $F$, then for all
$g\in\exp(\Lambda)$, $(L_g)_* \mathbf f$ is a unitary Hodge frame at $g\cdot F$
and this gives a holomorphic unitary frame on the entire orbit. Hence the
Chern connection is identically zero. This also follows immediately from
the formula for the connection form given above.
\qed\endproof
\medskip


\begin{thm}\label{thm:hol-curv} Let $D$ be a period domain for mixed Hodge
graded-polarized structures.
Let $\nabla$ be the Chern connection for the Hodge metric on the holomorphic
tangent bundle $T(D)$ at $F$. Then for all tangent vectors
$u  \in T^{1,0}_F(D)\simeq \eq$ we have
\begin{equation*}
\label{eqn:EndCalc}
   R_\nabla (u,\bar u) =
   - [(\ad{{\bar u }_+^*})_\eq ,(\ad{{\bar u }_+})_\eq]
  { -  \ad{[ u, \bar u]_0}
   - \left(\ad{([u,\bar u]_+ + [u,\bar u]_+^*)} \right)_\eq}.
\end{equation*}
We use the following convention: for all $u  \in\eg$ we write
$u  ^*_0,u  _+^*,u  _-^*$ to mean: first project onto
$\eg^{0,0}$, respectively $\eun_+$, $\eun_-$ and \emph{then} take the adjoint.
\end{thm}
\proof
Apply the formula of  Lemma~\eqref{lemma:curvature}. Proceeding as in the
proof of that Lemma, choose a complex curve $u(z)$ tangent to
$u\in T_FD$ and write $H(u(z))=H(z,\bar z)$.  We view the curve $u(z)$ as an
element of $\eq$, i.e.,  in the
preceding expression we replace $u$ by $zu  $ and $\bar u$ by $\bar z \bar u  $.
Then from Corollary~\ref{corr:H-form}  we have 
$\del_z H(0)=- (\ad{ (\bar u  )_+^*} )_\eq$, $\del_{\bar z} H(0)=-( \ad{(\bar u  )_+})_\eq$
and
\[
\begin{array}{lcl}
 \del_z\del_{ \bar z} H(0)&=& \half[(\ad{(\bar u  )_+^*})_\eq ,(\ad{(\bar u  )_+})_\eq]
   + \left(\ad{[u,\bar u]_0}\right)_\eq\\
&& \quad\quad +  \left(\ad{[\bar u,  u]_+ + [\bar u,u]_+^*} \right)_\eq.
\end{array}
\]
Since  at the point $F\in D$ we have
$R_\nabla (u,\bar u) = -\del_u  \del_{\bar u  } H  +\half [ \del_{ \bar u}  H ,  \del_{u}  H] $,
the result  follows.
\qed\endproof

\begin{rmk} \label{mtremark}(1) Note that in the pure case this gives back
$R_\nabla(u,\bar u)= - \ad{[u,\bar u]_0}$ as it should.\\
(2) By Remark~\ref{mMTD}, the  formula  for the curvature of a mixed Mumford--Tate domain is the same as  the one for the mixed period domain.
\\
(3) Exactly the same proof shows that the full curvature tensor, evaluated on pairs of tangent vectors $\set{u,v}\in T^{1,0}_FD$ is given by %
\[
\begin{array}{lll} R_\nabla  (u,\bar v) &=& -   ( [(\ad{\bar u^*_+})_\eq , (\ad{\bar v_+})_\eq ] )\\
         & \hphantom{=} & -\half( \ad{[u,\bar v]_0}+ \ad{[v,\bar u]_0})    \\
         & \hphantom{=} &  \qquad\quad + \left( \ad{ ( [\bar v, u]_+}   +[ \bar u ,v]_+^*   )\right)_\eq.
\end{array}
 \]
Alternatively, one may use
  \cite[(5.14.3)]{deligne}.   In that formula $R(u,v)$ stands for  the curvature in any pair $(u,v)$ of complex directions.
So $R(u,v)=  R_\nabla(u^{1,0},v^{0,1})-R_\nabla(v^{1,0},u^{0,1})$.
\end{rmk}


\section{Holomorphic Sectional Curvature in Horizontal
Directions}\label{sec:holseccurv}
Recall that the holomorphic sectional curvature is given by
\begin{equation}
           R(u):=  h(R_{\nabla}(u,\bar u)u,u)/h(u,u)^2   \label{eqn:hol-curv}.
\end{equation}
Our aim is to prove:
\begin{thm} Let $u\in T_F(D)$ be a \emph{horizontal vector} of unit length.
Then $R(u)=A_1+A_2+A_3+A_4$ where
\begin{eqnarray*}
A_1  &= & - \|[\bar u_+,u]_{\eq}\|^2,\\
A_2 & = &    \|[\bar u_+^*,u]_{\eq}\|^2,\\
A_3 &= &-h([[u,\bar u]_0,u],u)\\
A_4 &=& -h([[u,\bar u]_+,u]_{\eq},u) -h(u,[[u,\bar u]_+,u]_{\eq}).
\end{eqnarray*}
Each of these terms is real.
\end{thm}
\begin{proof}
%
We start by stating the following  two  self-evident basis principles which
can be used to simplify \eqref{eqn:hol-curv}:
\begin{itemize}
\item Orthogonality: The decomposition $\eg=\bigoplus \eg^{p,q}$ is orthogonal
for the Hodge metric;
 \item  Jacobi identity:  For all $X,Y,Z\in \End(V)$ we have
 $$
        [X,[Y,Z]] = [[X,Y],Z] + [Y,[X,Z]].
 $$
\item Metric conversion:  The relation
\begin{equation} \label{eqn:MetConv}
h([X,Y],Z) = h(Y,[X^*,Z])
\end{equation}
 implies
\begin{equation}
\label{eqn:MetTrick}
-h(\ad{[X  ,X  ^*]}Y ,Y )= \| [X  ,Y ]\|^2-\| [X  ^*,Y ]\|^2
\end{equation}
\end{itemize}

Theorem~\ref{thm:hol-curv} and the previous rules imply:
$$
\aligned
     h(R_{\nabla}(u,\bar u)u,u)
     &=  - h([( \ad{ (\bar u_+)^*})_\eq ,(\ad{(\bar u)_+})_\eq]u,u)
        -h((\ad{[u,\bar u]_0})u,u) \\
     &\qquad  -h((\ad{[u,\bar u]_+})u,u)
        -h((\ad{[u,\bar u]_+^*})_{\eq} u,u) \\
     &=-  \|[\bar u_+,u]_{\eq}\|^2 +  \|[\bar u_+^*,u]_{\eq}\|^2
        -h([[u,\bar u]_0,u],u) \\
     &\qquad -h([[u,\bar u]_+,u]_{\eq},u) -h(u,[[u,\bar u]_+,u]_{\eq}).
\endaligned
$$
This shows that  $h(R_{\nabla}(u,\bar u)u,u) =  A_1+A_2+A_3 + A_4$ where the terms $A_j$ are as stated.  In particular, the
terms  $A_1,A_2,A_4$ are real.
Metric conversion allows us to show that $A_3$ is real: 
\changed{}{setting $\alpha=u^{-1,1}$ and recalling \eqref{eqn:piz}, we see that
$[u,\bar u]_0 = [\alpha, \pi_\ez \bar\alpha]=[\alpha, \alpha^*]$ and so we find }that
\begin{equation}\label{eqn:a3}
\aligned
         A_3 &= -h([[\alpha,\alpha^*],u],u))            \\
             &= \|[\alpha,u] \|^2 - \| [\alpha^*,u]\|^2     \in \R.      \hfill  \qed
\endaligned
\end{equation}
\end{proof}

%

\par The next result gives the refinement of the curvature calculations
with respect to the decomposition of a horizontal vector into its Hodge
components:

\begin{thm} \label{MainResultonHolSectCurvature}For  $u= \sum_{j \le 1}u^{-1,j}\in\eg_\C$  set \footnote{Recall the notation
\eqref{eqn:piz}.}
\[
\aligned
      \alpha           &=& u^{-1,1},\quad \beta &=&  u^{-1,0} ,\qquad \qquad \lambda  &=& \sum\nolimits_{j\ge 1} u^{-1,-j}  \\
      \bar\alpha_+ &=&  \alpha^* +\epsilon,\quad \alpha^* &=&  \pi_\ez\bar \alpha _+ =  \bar\alpha_+^{1,-1}, \qquad \epsilon  &=&  \sum\nolimits_{j\ge 2}\bar\alpha_+^{0,-j}.
\endaligned
\]
Then,
\begin{eqnarray*}
A_1  &= &   -   \left(\|[\bar\beta_++\epsilon, \alpha ]\|^2+ \|[\bar\beta_++\epsilon,\beta]\|^2+\|[\bar\beta_++\epsilon,\lambda]_\eq\|^2\right),\\
A_2 & = &    \|[\alpha,\beta]\|^2+\|[\alpha,\lambda]\|^2+\|[\bar\beta^*_+,\beta]_\eq \|^2+\|[\bar\beta^*_++\epsilon^*,\lambda]\|^2,\\
A_3 &=
             & \|[\alpha,\beta] \|^2 +\|[\alpha,\lambda] \|^2 - \| [\alpha^*,\alpha]\|^2- \| [\alpha^*,\beta]\|^2   - \| [\alpha^*,\lambda]\|^2,  \\
A_4 &=& -2\|[\alpha^*,\lambda]\|^2 -2\|[\alpha^*,\beta]\|^2 + R(\alpha,\beta,\lambda),
\end{eqnarray*}
where
\[
R(\alpha,\beta,\lambda)=  -2\text{\rm Re} \left (h([[\lambda,\alpha^*],\lambda],\lambda)+
           h([[\alpha^*,\beta],\lambda],\lambda)+h([[\alpha^*,\lambda],\beta],\lambda)\right) .
\]
This last term vanishes if $\lambda$ has pure type. \\
Moreover, in the $\R$--split situation  we have $\bar\alpha^+=\alpha^*$ so that $\epsilon=0$.
\end{thm}
\begin{proof}
\textbf{The term $A_3$.} Inserting $u=\alpha+\beta+\lambda$ in \eqref{eqn:a3}
immediately gives the $A_3$-term.\\
\textbf{The terms $A_1, A_2$.}  We start by noting that   $\bar u_+ = \bar\alpha_++\bar \beta_+=\alpha^*+\epsilon+ \bar \beta_+$ and  so (note the precedence of the operators!)  $\bar u_+^*= \alpha+\epsilon^*+\bar \beta^*_+$.  Accordingly,
$$
        [\bar u_+,u]_{\eq} = [\alpha^*+\bar\beta_++\epsilon ,u]_\eq,\qquad
        [\bar u_+^*,u]_{\eq} = [\alpha+\bar\beta^*_++\epsilon^*,u].
$$
The first expression gives
$$
\aligned
       A_1 =& -  \| [\bar\beta_++\epsilon ,u]\|^2\\
              =& -  (\|[\bar\beta_+ +\epsilon,\alpha]\|^2
                                      +\|[\bar\beta_++\epsilon,\beta]\|^2 + \|[\bar\beta_++\epsilon,\lambda]\|^2).
\endaligned
$$
by orthogonality.
The   second expression expands as:
$$
      [\bar u_+^*,u]_{\eq}
      =   [\alpha,u]
        + [\bar\beta^*_++\epsilon^*,\alpha]_{\eq} + [\bar\beta^*_++\epsilon^*,\beta+\lambda]_{\eq}
$$
For weight reasons, $[\bar\beta^*_+,\alpha]_{\eq} = 0$ and $[\epsilon^*,\alpha]_\eq=[\epsilon^*,\beta]_\eq=0$.  Therefore, by orthogonality:
$$
     A_2 =  \|[\bar u^*,u]_{\eq}\|^2
     = \| [\alpha,\beta]\|^2 +  \| [\alpha,\lambda]\|^2 + \|[\bar\beta^*_+,\beta]_{\eq}\|^2 + \|[\bar\beta^*_+,\lambda]_{\eq}\|^2+\|[\epsilon^*,\lambda]_\eq\|^2 .
$$
\textbf{The term  $A_4$.} To calculate $A_4$, we observe that
$$
      [u,\bar u]_+ =[\beta,\bar\alpha_+]+  [\lambda,\bar\alpha_+]=  [\beta, \alpha^* +\epsilon ]+  [\lambda,\alpha^*+\epsilon].
$$
So $h([[u,\bar u]_+,u],u) = h([[\beta,\alpha^*+\epsilon],u],u)+ h([ [\lambda,\alpha^*+\epsilon],u],u)$ and we consider each term separately.
For the  first term, note that  $[[\beta,\epsilon],u]$ as well as $ [[\lambda,\epsilon],u] $ belong to $\bigoplus_{j\ge 1} \eg^{-2,-j}$ and hence  are both orthogonal to $u$ and we can discard these terms. Moreover,  $[\beta,\alpha^*]\in \eg^{0,-1}$ and so, by orthogonality,
$$
\aligned
    h([[\beta,\alpha^*],u],u) &=   h([[\beta,\alpha^*],\alpha],\beta) +  h([[\beta,\alpha^*],\beta],\lambda)+h([[\beta,\alpha^*],\lambda],\lambda).
\endaligned
$$
Since
$
       -h([\alpha,[\beta,\alpha^*]],\beta)
        = -h([\beta,\alpha^*],[\alpha^*,\beta])
        = \|[\alpha^*,\beta]\|^2
$
we find for the first term
 $$
\aligned
    h([[\beta,\alpha^*],u],u)  =   \|[\alpha^*,\beta]\|^2+  h([[\beta,\alpha^*],\beta],\lambda)+h([[\beta,\alpha^*],\lambda],\lambda).
\endaligned
$$
Note that   $[\lambda,\alpha^* ]\in \bigoplus_{j\ge 0}\,\eg^{0,-2-j}$ so that by orthogonality,
$$
\aligned
       h([[\lambda, \alpha ^*],u],u)
       &=  h([[\lambda, \alpha^*],\lambda],\beta)  + h([[\lambda, \alpha^*],\alpha + \lambda],\lambda) .
      \endaligned
$$
The second term thus simplifies to
$$
\aligned
       h([[\lambda,\alpha^*],\alpha],\lambda)
          +  h([[\lambda,\alpha^*],\lambda],\lambda)
        &= -h([\alpha,[\lambda,\alpha^*]],\lambda)
          +  h([[\lambda,\alpha^*],\lambda],\lambda)                 \\
        &= -h([\lambda,\alpha^*],[\alpha^*,\lambda])
           +  h([[\lambda,\alpha^*],\lambda],\lambda)                \\
        &= \|[\alpha^*,\lambda]\|^2
           +  h([[\lambda,\alpha^*],\lambda],\lambda)   .
\endaligned
$$
It follows that
$$
\aligned
       A_4 & = -2\|[\alpha^*,\lambda]\|^2-2\|[\alpha^*,\beta]\|^2    \\
       & \hspace{4em}-\text{\rm Re} \left (h([[\lambda,\alpha^*],\lambda],\lambda)+
           h([[\alpha^*,\beta],\lambda],\lambda)+h([[\alpha^*,\lambda],\beta],\lambda)\right). \hfill\qed
           \endaligned
$$
\end{proof}

\begin{rmk}\label{rmk:epsilon} We claim that   $\epsilon$ and the
Deligne splitting $\delta$ of $(F,W)$ are related as follows:
\[
\epsilon= [-2\ii\delta,\bar\alpha]_+.
\]
To see this, apply the Deligne splitting:
$$
         \alpha = \Ad(e^{\ii \delta})\alpha^{\ddag}
$$
where $\alpha^{\ddag}$ is type $(-1,1)$ at the split mixed Hodge structure
$(\hat F,W)$ defined by $\hat F = e^{-\ii\delta}F$.   At that point  the complex conjugate and the adjoint of $\alpha^{\ddag}$ coincide. Therefore,
\[
\aligned
\alpha^* &= \Ad (e^{\ii\delta} ) [  {\alpha^{\ddag}}^*]_{\hat F} = \Ad (e^{\ii\delta} )[ \overline{ \alpha^{\ddag}}]_{\hat F}\\
 \bar\alpha &= \Ad(e^{-\ii\delta})[\overline{\alpha^{\ddag}} ]_{\hat F}
                   = \Ad(e^{\ii\delta})[ \Ad(e^{-2\ii\delta})\overline{\alpha^{\ddag}}]_{\hat F}.
\endaligned
\]
Consequently,
$$
\aligned
       \epsilon &= (\alpha^* -\bar\alpha)_+\\
       &= \Ad(e^{\ii\delta})
                    ((\Ad(e^{-2\ii\delta})-1)\overline{\alpha^{\ddag}})_{+,\hat F} \\
                &= \Ad(e^{\ii\delta})
                    [-2i\delta,\bar\alpha^{\ddag}]_{+,\hat F} \\
                &=  [-2\ii\delta,\Ad(e^{\ii\delta})\overline{\alpha^{\ddag}}]_+ \\
                &=  [-2\ii\delta,\Ad(e^{2\ii\delta})\bar\alpha]_+ \\
                &=  [-2\ii\delta,\Ad(e^{2\ii\delta})\bar\alpha]_+\\
                &=  [-2\ii\delta,\bar\alpha]_+.
\endaligned
$$
\end{rmk}

\par We shall now discuss particular cases.
\begin{corr}  \label{hsecex2} The holomorphic sectional
curvature along a horizontal direction $u = \alpha + \lambda$ with
$\alpha$ of  type $(-1,1)$ and $\lambda\in\Lambda$ equals
$$
R(u)= \frac{ 2\|[\alpha,\lambda]\|^2  +f(u,\epsilon)
       - 3\|[\alpha^*,\lambda]\|^2
       - \|[\alpha,\alpha^*]\|^2
       - \text{\rm Re} (h([[\lambda,\alpha^*],\lambda],\lambda))}{(\|\alpha\|^2+\|\lambda\|^2)^2},
$$
where $f(u,\epsilon)=  -\left(\|[\alpha,\epsilon]\|^2+\| [\lambda,\epsilon] \|^2\right)+\| [\lambda,\epsilon^*]\|^2$.
In particular:
\begin{itemize}
\item   $R(u)\le 0$ if $[\alpha,\lambda]=0=[\lambda,\epsilon^*]$ and $\lambda$ is of pure type
$(-1,-k)$ for some $k<0$ (since
$[[\lambda ,\alpha^*],\lambda]$ and $\lambda$ have different types),
and $R(u)<0$  as soon as $\alpha\neq 0$.
\item $R(u)>0$  if  $[\alpha^*,\lambda]=0=[u,\epsilon]$
provided
$2\|[\alpha,\lambda]\|^2+  \| [\lambda,\epsilon^*]\|^2  >\|[ \alpha^*,\alpha]\|^2$.
\end{itemize}
\end{corr}


\begin{exmple}\label{exmple:kahler-def} Let us return to the setting
of the variation of mixed Hodge structure \eqref{eqn:kahler-def} arising
from a variation of K\"ahler moduli along a family of compact K\"ahler
manifolds.   The original variation $F(s)$ of a direct sum of pure Hodge
structures that can be expressed locally as
$$
         F(s) = e^{\Gamma(s)} \cdot F(0)
$$
where $\Gamma:\Delta^r\to\eq$ vanishes at $0$ and takes values in
$\eg^{-1,1}\oplus\eg^{-2,2}\oplus\cdots$.  The requirement that each
$\gamma_j$ be of type $(-1,-1)$ for all $F(s)$ implies that
$$
            \Ad(e^{-\Gamma(s)})\lambda_j = e^{-\ad\Gamma(s)}\lambda_j
$$
is horizontal at $F(0)$ for all $s$.  Via differentiation along a holomorphic
arc through $s=0$, this fact implies that $[\Gamma'(0),\gamma_j] = 0$
since $\Gamma'(0)\in\eg^{-1,1}$ and
$\Gamma(0) = 0$.

\par The local normal form of the variation \eqref{eqn:kahler-def} is
therefore
$$
         \tilde F(s) = e^{\ii N(u(s))}e^{\Gamma(s)} \cdot F(0)
$$
where $u(s)$ takes values in the complex linear span $L_{\mathbb C}$ of
$\gamma_1,\dots,\gamma_k$.  Accordingly, the derivative of $(\tilde F(s),W)$
at $s=0$ is
$$
         \xi = \xi^{-1,-1} + \xi^{-1,1},\qquad
         \xi^{-1,-1} =\ii  N(u'(0)),\qquad \xi^{-1,1} = \Gamma'(0)
$$
where $[\xi^{-1,-1},\xi^{-1,1}]=0$. 
Recall the statement of Theorem~\ref{MainResultonHolSectCurvature} for the definition of $\epsilon$.
We show that it vanishes in this situation.  First observe that since the
\changed{}{"untwisted"} mixed Hodge structure
$(F(s),W)$ are all split over $\R$, the element $\delta$ attached to
$(\tilde F(0),W)$ is defined by the equation
$$
     e^{-\ii N(\overline{u(0)})} \cdot Y_{(F(0),W)} =
     e^{-2\ii \delta}e^{\ii N(u(0))} \cdot Y_{(F(0),W)}.
$$
Since $\delta$ commutes with all $(p,p)$-morphisms of $(\tilde F(0),W)$,
it follows from the previous equation that $\delta = N(\re(u(0)))$.
Accordingly, $\delta$ is real and belongs to $L_{\C}$ and so
$$
        [\bar\Gamma'(0),\delta] = \overline{[\Gamma'(0),\delta]} = 0.
$$
By Remark \eqref{rmk:epsilon}, it follows that indeed $\epsilon=0$.
Corollary~\ref{hsecex2}  then implies:
\[
R(\xi)\le 0  \text{  and  }  <0  \text{ if  } \xi\not=0.
\]
\end{exmple}

\begin{corr}\label{curvapll2} The holomorphic sectional
curvature along a horizontal direction $u = \alpha + \beta$ with $\alpha$ type
$(-1,1)$ and $\beta$ type $(-1,0)$ is
$$
\aligned
R(u)&=  \frac{-n(\alpha,\beta)+p(\alpha,\beta)
 } {(\|\alpha\|^2+\|\beta\|^2)^2},\\
 n(\alpha,\beta) &:=   \|[\alpha^*+\epsilon,\alpha]\|^2 + \|[\epsilon,\beta]\|^2
 + 3 \|[\alpha^*,\beta]\|^2 +\|[\alpha,\bar\beta_+] \|^2
 + \|[\bar\beta_+,\beta]\|^2,\\
 p(\alpha,\beta)&:=  \|[\alpha,\beta]\|^2
 +\| [\bar\beta^*_+,\beta]_{\eq}\|^2.
\endaligned
$$
In particular, if $\alpha=0$,   $[\beta,\bar\beta_+]=0=[\epsilon,\beta]$ (which is the case if $W_{-1}\eg_\C$ is abelian) we have $ {R}(u)\ge   0$.
\end{corr}

Next, we look at  a  unipotent variation
of mixed Hodge structure in the sense of Hain and Zucker \cite{unipotvars}.
These are the variations where the pure Hodge structures on the graded
quotients are constant so that $\alpha=u^{-1,1}=0$ and hence $\epsilon=0$.
This situation  occurs in two well known geometric
examples:
\begin{itemize}
\item The VMHS on $J_x/J_x^3$, $x\in X$ where $X$ is a smooth complex
projective variety;
\item The VMHS attached to a family of homologically trivial
algebraic cycles moving in a fixed variety $X$.
\end{itemize}

\begin{corr} 
For the curvature coming from a unipotent variation we
have $$
 R(u)=   \frac{-  \|[\bar\beta_+,\beta]\|^2
   - \| [\bar\beta_+,\lambda]\|^2
    + \| [\bar\beta^*_+,\beta]_{\eq}\|^2
    +\| [\bar\beta^*_+,\lambda]_\eq\|^2  }{(\|\beta\|^2+\|\lambda\|^2)^2}.
$$
\end{corr}

\section{Curvature of Hodge Bundles} \label{sec:hb}

\subsection{Hodge Bundles over Mixed Period Domains}
In this subsection, we
compute the curvature of the Hodge bundles over the classifying
space $D$ using the methods of \S~\ref{ssec:secondorder}.  Since the
Hodge bundles of a variation of mixed Hodge structure $\VV\to S$
are obtained by pulling back the Hodge bundles of $D$ along local
liftings of the period map, this furnishes a computation of the
curvature of the Hodge bundles of a variation of mixed Hodge structure.

\par Let $F\in D$ and $\eq$ be the associated
nilpotent subalgebra \eqref{eqn:tang-alg} and $U$ be a neighborhood of
zero in $\eq$ such that the map $u\to e^u\cdot F$ is a biholomorphism onto a
neighborhood of $F$.  Then, we obtain a local holomorphic framing for
the bundle $\FF^p$ over $U$ via the sections $\alpha(u) = e^u\alpha$
for fixed $\alpha\in F^p$.  Let $\beta(u) = e^u\beta$ be another such
section of $\FF^p$ over $U$, and $L_g$ denote the linear action of
$g\in GL(V_{\C})$ on $V_{\C}$.  Let $\Pi$ denote orthogonal projection
from $V_{\C}$ to $F^p$.  Then, as in \S~\ref{ssec:secondorder} by \eqref{eqn:expu}, the metric is
$$
\aligned
      h_{e^u\cdot F}(\alpha(u),\beta(u))
        &= h_{F}(L_{\exp(\varphi(u))}\alpha,L_{\exp(\varphi(u))}\beta)  \\
        &= h_{F}(\Pi\comp  L_{\exp(\varphi(u))}\alpha,L_{\exp(\varphi(u))}\beta) \\
        &= h_{F}(L_{\exp(\varphi(u)^*)}\Pi\comp  L_{\exp(\varphi(u))}\alpha,\beta) \\
        &= h_{F}(\Pi\comp  L_{\exp(\varphi(u)^*)}
                  \Pi\comp  L_{\exp(\varphi(u))}\alpha,\beta).
\endaligned
$$
In analogy with \S~\ref{ssec:tang}, we have the identity
$$
            \Pi\comp  L_{\exp(\varphi(u))} = L_{\exp(\Pi\comp \varphi(u))},
$$
since $\varphi(u)$ belongs to the subalgebra preserving $F^p$.
The identity
$$
           \Pi\comp  L_{\exp(\varphi(u)^*)} = L_{\exp(\Pi\comp \varphi(u)^*)}
$$
is also straightforward because $\varphi(u)$ is a sum of components
of Hodge type $(a,b)$ with $a\geq 0$.  As such $\varphi(u)^*$ is a sum
of components of Hodge type $(-a,-b)$ with $-a\leq 0$, and hence there is no
way for the action of $\varphi(u)^*$ to move a vector of Hodge type $(c,d)$
with $c<p$ back into $F^p$.

\par Accordingly, by the universal nature of the Campbell--Baker--Hausdorff
formula, the only difference between the computation of the curvature of
$\FF^p$ and the curvature of $T(D)$ is that for the former we are use the
linear action $GL(V_{\C})$ and $\eg\el(V_{\C})$ and project orthogonally to $F^p$
whereas in the later we use the adjoint action and project orthogonally to
$\eq$.   So, with $\Pi$ the  orthogonal projection from $V_{\C}$ to
$F^p$ for  $u,v\in T^{\rm hol}_{F}(D)$  we find
\[
\aligned
    R_\nabla  (u,\bar v) &=  -   ( [\Pi\comp (\bar u^*_+), \Pi\comp (\bar v_+)] )\\
         & \hphantom{=}  \hspace{3.5em} -\half
                          \left( \Pi\comp ([u,\bar v]_0)
                                  + \Pi\comp ([v,\bar u]_0)\right)    \\
         & \hphantom{=}    \hspace{7em}
         + \Pi\comp \left([\bar v, u]_+ +[\bar u ,v]_+^* \right) .
\endaligned
 \]
Taking account of the fact that the terms with subscript $+$ (without
an adjoint) and subscript $0$ always preserve $F^p$ this simplifies
and we get:
\begin{corr} Let $\Pi$ denote orthogonal projection from $V_{\C}$ to
$F^p$.  Then, the curvature of the Hodge bundle $\FF^p$ over $D$ in the
directions $u,v\in T^{\rm hol}_{F}(D)$ is
\[
\aligned
    R_\nabla  (u,\bar v) &=  -   ( [\Pi\comp (\bar u^*_+), \bar v_+] )\\
         & \hphantom{=}  \hspace{2.5em}  -\half
                          \left([u,\bar v]_0 + [v,\bar u]_0\right)    \\
         & \hphantom{=}    \hspace{5em}
         + \left([\bar v, u]_+ + \Pi\comp [\bar u ,v]_+^* \right) .
\endaligned
 \]
\end{corr}

\par The computation of the curvature of the quotient bundle $\FF^p/\FF^{p+1}$
proceeds along the same lines as the computation of the curvature of $\FF^p$.
However, in this case the corresponding projection operator $\Pi'$ sends
$V_{\C}$ to
$$
           \FF^p/\FF^{p+1}\cong U^p := \bigoplus_q\,\II^{p,q}_{(F,W)}.
$$
The identity
$$
            \Pi'\comp  L_{\exp(\varphi(u))} = L_{\exp(\Pi'\comp \varphi(u))}
$$
results from the fact that elements of $\eg_{\C}^{F}$ have Hodge components of
type $(a,b)$ with $a\geq 0$ and such an element moves $U^p$ to $U^{p+a}$.
A similar argument works for $\Pi'\comp \varphi(u)^*$.

\begin{corr}\label{corr:hodge-bundle} Let $\Pi'$ denote orthogonal projection
from $V_{\C}$ to $U^p$ at $F$.  Then, the curvature of the Hodge bundle
$\FF^p/\FF^{p+1}$ over $D$ in the directions $u$, $v\in T^{\rm hol}_{F}(D)$ is
\[
\begin{array}{lll}
    R_\nabla  (u,\bar v) &=& -   ( [\Pi'\comp (\bar u^*_+),
                                  \Pi'\comp (\bar v_+)] )\\
         & \hphantom{=} & -\half
                          \left( \Pi'\comp ([u,\bar v]_0)
                                  + \Pi'\comp ([v,\bar u]_0)\right)    \\
         & \hphantom{=} &  \qquad\quad
         + \Pi'\comp \left([\bar v, u]_+ +[\bar u ,v]_+^* \right).
\end{array}
 \]
Taking account of the fact that the terms with subscript $ 0$  preserve
$U^p$ it follows that
\[
\begin{array}{lll}
    R_\nabla  (u,\bar v) &=& -   ( [\Pi'\comp (\bar u^*_+),
                                  \Pi'\comp (\bar v_+)] )\\
         & \hphantom{=} & -\half
                          \left([u,\bar v]_0 + [v,\bar u]_0\right)    \\
         & \hphantom{=} &  \qquad\quad
         + \Pi'\comp \left([\bar v, u]_+ +[\bar u ,v]_+^* \right).
\end{array}
 \]
\end{corr}

\subsection{First Chern Forms and Positivity}
Let us calculate the first Chern form of the Hodge bundles $\UU^p$ over a disk  $\Delta:f \to D$ with local coordinate $s$.
Set $f(s)=F_s$ and $u=f_*(d/ds)_{F_s}$. We also let
\[
u^{(p)}: \UU^p \to \UU^{p-1},\quad u^{p}= \alpha^{(p)}+\beta^{(p)}+\lambda^{(p)}
\]
be the restriction of $u$ to $\UU^p$ and $\alpha^{(p)},\beta^{(p)}$ and $\lambda^{(p)}$ the decomposition into types $(-1,1), (-1,0)$, respectively $\sum_{k\ge 1} (-1,-k)$. Then we have
\begin{lemma} \label{chrnformHodgebundles} The first Chern form $c_1(\UU^{p})$ involves only the components $\alpha^{(p)}$  of $u$ of type $(-1,1)$ and locally can be written
\[
c_1(\UU^{p}) = \frac{1}{2\pi \ii} \, \left(  \|| \alpha^{(p)}\||_{F_s} -   \|| \alpha^{(p+1)}\||_{F_s}   
  \right)  ds\wedge d\bar s.
  \]
\end{lemma}
\proof
We have to calculate $\tr   R_\nabla  (u,\bar u)$ using Cor.~\ref{corr:hodge-bundle}.  Let us write $u= \alpha+\beta+\lambda$ as before. Since  $\Pi'\comp (\bar u_+) = \bar\beta_+ $, we find
\begin{eqnarray}
 \left[ \Pi'\comp (\bar u_+^*), \Pi'\comp (\bar u_+) \right]   &=  & [\bar\beta^*_+,\bar\beta_+] \label{eqn:hc1} \\
\left[  u, \bar u \right]_0&=&  [\alpha,  \alpha^*]   \label{eqn:hc2} \\
 \Pi'\comp [\bar u, u]_+  &=&  [ \alpha^*,\beta+\lambda] \label{eqn:hc3} .
 \end{eqnarray}
 The first  two terms  preserve the bi-degree but this is not the case for \eqref{eqn:hc3}. So, computing traces, we can discard it.
The vanishing of the trace of $[\bar\beta_+^*,\bar\beta_+]$ follows from the
standard calculation
$$
     \tr ([A^*,A]) = \tr (A^*A) - \tr (AA^*) = \tr (AA^*) - \tr (AA^*) = 0
$$
with $A= \bar\beta_+\in\text{End}(\mathcal U^p)$.  On the other hand,
since $\alpha$ maps $\mathcal U^p$ to $\mathcal U^{p-1}$ this argument
does not apply \eqref{eqn:hc2}, and so
%
\[
\aligned
 \tr   R_\nabla  (u,\bar u) &=     - \tr  [\bar\beta^*_+,\beta]\, |  \UU^{p} -  \tr  [\alpha,  \alpha^*] \, |\UU^{p}\\
 &=  \|| \alpha^{(p)}\||_{F_s} -   \|| \alpha^{(p+1)}\||_{F_s}. \qed
 \endaligned  \]
\endproof
\begin{corr}\label{corr:top-hodge}
The "top"  Hodge bundle, say $\UU^n\simeq \FF^n$ (which is a holomorphic sub bundle of the total bundle) has a  non-negative  Chern form:
\[
c_1(\UU^{n}) = \frac{\ii }{2\pi} \, \left(  \|| \alpha^{(n)}\||_{F_s}  
\right)  ds\wedge d\bar s\, \ge \,0 .
\]
\end{corr}
As in \cite[Prop. 7.15]{periods2} one deduces form Lemma~\ref{chrnformHodgebundles} also:
\begin{corr} Let $\EE^p:= \FF^p/\FF^{p+1}$ and put
\[
K(\FF^\bullet):= \bigotimes _p (\det (\EE^p))^{\otimes p}.
\]
Then the first Chern form of $K(\FF^\bullet)$ is non-negative and is zero precisely in the horizontal directions $(-1,k)$ with $k\le 0$.
\end{corr}
Let us now consider the curvature form itself.
\begin{exmple} Consider the case with two adjacent weights $0\subset W_0\subset W_1=V$. Split the top Hodge bundle
as $\FF^n= \II^{n,-n}\oplus \II^{n,-n+1}$ and decompose the curvature matrix accordingly
\[
R(u,\bar u) =\begin{pmatrix}
\alpha^*\comp\alpha  + \bar\beta\comp \bar\beta^* & \alpha^*\comp \beta\\
-\beta^*\comp \alpha&  \alpha^*\comp\alpha  - \bar\beta^*\comp \bar\beta
\end{pmatrix},\, u=\alpha+\beta.
\]
We see that for $v\in V_\C$,  $\|R(v) (u,\bar u)\|_F= \Tr \bar v R (u,\bar u)  v  \ge 0$ if $u=\alpha$, but $\|R(v) (\beta,\bar \beta)\|_F =\|\bar\beta^*(v^{(-n)}) \|_F- \| \bar\beta(v^{(-n+1)}) \|_F$ which need not be $\ge 0$.  
\end{exmple}
From the preceding example it follows that we can expect positive curvature at most in the $\alpha$-direction. In fact, this is true:
\begin{prop} \label{curvtopHodge} The "top"  Hodge bundle, say $\UU^n\simeq \FF^n$  has  a positive  curvature   in the $\alpha$-directions and has  identically zero curvature in the $\lambda$-directions.
\end{prop}
\proof
We note the diagonal terms in the curvature form involve $\alpha^{(q)}\comp (\alpha^{(q)})^*$ acting on $\II^{n,q}$.
Let $r$ be the minimal $q$ with $\II^{n,q}\not=0$ and consider the splitting $\UU^n =\II^{n,r}\oplus\II^{n,r+1}\oplus  \II^{n,>r+1}$.
Assume $\beta=0$.
The matrix of the curvature form splits accordingly:
\[
 R(u,\bar u) =
\begin{pmatrix}
\alpha^* \comp \alpha   &0  & \alpha^*\comp \lambda \\
 0     &\alpha^* \comp \alpha  & 0 \\
-\lambda^*\comp\alpha&  0  & \alpha^* \comp \alpha
\end{pmatrix},\, u=\alpha+\lambda.
\]
So with $v\in \UU^{n}$ one finds  for $u=\alpha+\lambda$:
\[
R(v)(u,\bar u)= \|\alpha(v)\|_F ^2 \ge 0
\]
with equality if  $\alpha(v)=0$.  \qed\endproof

Here is an example of a variation  where  $\beta=0$:
\begin{exmple}
Consider \emph{higher normal functions associated to motivic cohomology} $H^p_{\MM} (q)$, see \cite{3authors}.
Indeed, these give extension of $R^{p-1}\pi_*\Z(q)$ with $p-2q-1<0$ where $\pi: X\to S$ is a smooth projective family.
\\
Assume moreover that  the cohomology $H^{p-1}(X_t)$ of the fibres $X_t$ is such that the non-zero Hodge numbers are $h^{p-1-q,q},\cdots h^{q, p-1,q}$ (i.e. the Hodge structure has level $=p-1-2q$).  With $n=2q+1-p$ the non-zero Hodge numbers of the mixed variation are, besides $h^{0,0}$ indeed precisely $h^{-n,0},\dots ,h^{0,-n}$. Here $\beta=0$ while $\lambda\not=0$.
\end{exmple}

\subsection{Variations of Mixed Hodge Structure}
 We want to stress that, although the above calculations are done on
the period domain, they apply also for variations of mixed Hodge structure:
the Hodge bundles simply pull back and so does the Hodge metric. What remains
to be done is to identify the actions of $u,v$ when these are tangent to period maps.

To do this and also as a check on the preceding  calculations, we shall now compute the curvature
of the Hodge bundles of a variation of mixed Hodge structure starting
from Griffiths computation for a \emph{variation of pure Hodge structure}
$\mathcal H$.  To this end, we recall that the Gauss--Manin connection $\nabla$ of
$\mathcal H$ decomposes as
$$
       \nabla= \bar\theta_0 + \underbrace{\bar\pd_0 + \pd_0}_{D} + \theta_0,
$$
where $\bar\pd_0$ and $\pd_0$ are conjugate differential operators of type
$(0,1)$ and $(1,0)$ respectively which preserve the Hodge bundles
$\mathcal H^{p,q}$, while $\theta_0$ is an endomorphism valued 1-form
which sends $\mathcal H^{p,q}$ to $\mathcal H^{p-1,q+1}\otimes\EE^{1,0}$
and $\bar\theta_0$ is the complex conjugate of $\theta_0$.  The
connection $D=\bar\pd_0 + \pd_0$ is hermitian with respect to the Hodge
metric:
$$
       dh(u,v) = h((\bar\pd_0  + \pd_0)u,v)
                 + h(u,(\bar\pd_0 + \pd_0)v).
$$
In particular, since $\bar\pd_0$ coincides with the induced action of the
$(0,1)$-part of the Gauss--Manin connection acting on
$$
             \mathcal H^{p,q} \cong \FF^p/\FF^{p+1},
$$
it follows that $D$ is the \emph{Chern connection}, i.e., the hermitian holomorphic connection
of the system of Hodge bundles attached to $\mathcal H$.  Expanding out
$$
        (\bar\theta_0 + \bar\pd_0 + \pd_0 + \theta_0)^2 = 0
$$
and decomposing with respect to Hodge types shows that
$$
        R_{D} = -(\theta_0\wedge\bar\theta_0
                              + \bar\theta_0\wedge\theta_0).
$$
If $d/ds$ is a holomorphic vector field on $S$,  the value $u$ of $\theta_0(f_*(d/ds))$ at zero belongs to $\eg^{-1,1}$   and $R_D(u,\bar u)= -[u, \bar u]$ which checks  with the previous calculation.

\par To compute the curvature of the Hodge bundles
$\FF^p/\FF^{p+1}$ of a variation of \emph{mixed} Hodge structure,
$\VV\to S$ we consider the $C^{\infty}$-subbundles $\UU^p$
obtained by pulling back $\UU^p \to D$ along the variation, i.e.
$$
                \II^{p,q}(s) = I^{p,q}_{(\FF(s),\mathcal W)},\qquad
                 \UU^p = \bigoplus_q\,\II^{p,q}.
$$
By \cite{higgs}, the Gauss--Manin connection of $\VV$
decomposes as
$$
                \nabla = \tau_0  + \bar\pd + \pd + \theta
$$
where $\bar\pd$ and $\pd$ are differential operators of type $(0,1)$
and $(1,0)$ which preserve $\UU^p$ whereas
$
      \theta:\UU^p \to \UU^{p-1}\otimes\EE^{1,0}$ and $      \tau_0:\UU^p \to\UU^{p+1}\otimes\EE^{0,1}$.
One has
\[
\aligned
 \II^{p,q} &\mapright{\tau_0} (\II^{p+1,q-1}\otimes \EE_S^{0,1}), \\
 \II^{p,q} &\mapright{\theta=(\theta_0,\theta_-)}  (\II^{p-1,q+1}\otimes \EE_S^{1,0}) \oplus( \oplus_{k\ge 2}\II ^{p-1,q+k} \otimes \EE_S^{1,0}).
\endaligned
\]
Similarly
$$
\aligned
  \II^{p,q} & \mapright{\pd} \II^{p,q}   \otimes \EE_S^{1,0},\\
   \II^{p,q} & \mapright{\bar\pd = (\bar\pd_0 , \tau_-)}  (\II^{p,q}   \otimes \EE_S^{0,1}) \oplus( \oplus_{k \ge 1}
 \II^{p,q-k}   \otimes \EE_S^{0,1}) .
 \endaligned
$$
To unify notation, we also write $\pd = \pd_0$.  Then, we have
$$
       \nabla = \tau_0 +  \tau_- + \bar\pd_0   +  \pd_0 + \theta_- + \theta_0
$$
In particular, relative to the $C^{\infty}$ isomorphism of $\gr^W_k$ with
$$
    \EE_k := \bigoplus_{p+q=k}\,\II^{p,q}
$$
the induced action of $\nabla$ on $\gr^W_k$ coincides with the action of
$$
            D_0 = \tau_0 + \bar\pd_0 + \pd_0 + \theta_0
$$
on $\EE_k$.  Given that the mixed Hodge metric is just the pullback
of the Hodge metric on $\gr^W_k$ via the isomorphism with $\EE_k$,
it follows that $\bar\pd_0 + \pd_0$ is a hermitian connection on
$\UU^p$.  In particular, since the induced holomorphic structure
on $\UU^p$ is given by $\bar\pd$ and by the adjoint property,  it follows that
\begin{equation}
        D =  \underbrace{\tau_-+\bar\pd_0}_{\bar\del} + \pd_0 - \tau_-^*    \label{eqn:hh-conn}
\end{equation}
is the Chern   connection of $\UU^p$ relative
to the mixed Hodge metric. Thus,
$$
         R_D = R_{(\bar\pd + \pd_0) - \tau_-^*}
             =  R_{(\bar\pd + \pd_0)} - (\bar\pd + \pd_0)\tau_-^*
                + \tau_-^*\wedge\tau_-^*.
$$
To simplify this, observe that $\tau_-^*$ is a differential form of
type $(1,0)$, so we must have
$$
        -\pd\tau_-^* +  \tau_-^*\wedge\tau_-^* = 0
$$
in order to get a differential form of type $(1,1)$.  Therefore,
$$
         R_D =  R_{(\bar\pd + \pd_0)} - \bar\pd\tau_-^* .
$$
Expanding out
$$
       \nabla^2  =
       (\tau_0 + \bar\pd + \pd_0 + \theta)^2 = 0,
$$
it follows that
\begin{equation}
        R_{(\bar\pd + \pd_0)} = -(\theta\wedge\tau_0 + \tau_0\wedge\theta)
                         \label{eqn:higgs-conn}
\end{equation}
and hence
$$
        R_D = -(\theta\wedge\tau_0 + \tau_0\wedge\theta) - \bar\pd\tau_-^* .
$$

\par To continue, we note that
$$
        \bar\pd\tau_-^* = (\bar\pd_0 + \tau_-)\tau_-^*
                        = \bar\pd_0\tau_-^* + \tau_-\wedge\tau_-^*
                                           + \tau_-^*\wedge\tau_-
$$
and so
\begin{equation}
        R_D = -(\theta\wedge\tau_0 + \tau_0\wedge\theta)
              -(\tau_-\wedge\tau_-^* + \tau_-^*\wedge\tau_-) - \bar\pd_0\tau_-^*.
        \label{eqn:D-curv}
\end{equation}
To finish the calculation, we differentiate the identity
$$
         h(\tau_-(\sigma_1),\sigma_2) = h(\sigma_1,\tau_-^*(\sigma_2))
$$
and take the $(1,1)$ part to obtain
$$
\aligned
     h((\pd_0 \tau_-)(\sigma_1)
        &+ \tau_-(\pd_0 \sigma_1),\sigma_2)
         + h(\tau_-(\sigma_1),\bar\pd_0\sigma_2) \\
        &= h(\pd_0\sigma_1,\tau_-^*(\sigma_2))
           + h(\sigma_1,(\bar\pd_0\tau_-^*)(\sigma_2)
                        +  \tau_-^*(\bar\pd_0 \sigma_2)).
\endaligned
$$
Using the properties of the adjoint, this simplifies to
\begin{equation*}
        \bar\pd_0\tau_-^* = (\pd_0\tau_-)^* .   
\end{equation*}
It remains to compute $\pd_0\tau_- = \pd\tau_-$. To do this, first  observe that
$$
        R_{\bar\pd + \pd}
         = R_{\bar\pd_0 + \pd_0 + \tau_-}
         = R_{\bar\pd_0 + \pd_0} + (\bar\pd_0 + \pd)\tau_-+ \tau_-\wedge\tau_-.
$$
Now note that   equation \eqref{eqn:higgs-conn} implies that $R_{\bar\pd + \pd}$ is of
type $(1,1)$, and hence
$$
      R_{\bar\pd + \pd_0} = R_{\bar\pd_0 + \pd_0} + \pd\tau_-,
$$
since $R_{\bar\pd_0 + \pd_0}$ is also of type $(1,1)$ as the curvature of
hermitian holomorphic connection for $h$ and $\bar\pd_0$.  Moreover,
since $\bar\pd_0 + \pd_0$ preserves the bigrading by $\mathcal I^{p,q}$
whereas $\pd\tau_-$ lowers weights, it follows from \eqref{eqn:higgs-conn}
that
$$
            \pd\tau_- = -(\theta_-\wedge\tau_0 + \tau_0\wedge\theta_-).
$$

\begin{corr} The curvature of the Hodge bundles of a variation of
mixed Hodge structure $\VV\to S$ is
$$
     R_D = -(\theta\wedge\tau_0 + \tau_0\wedge\theta)
           -(\theta_-\wedge\tau_0 + \tau_0\wedge\theta_-)^*
              -(\tau_-\wedge\tau_-^* + \tau_-^*\wedge\tau_-).
$$
\end{corr}

Let us compare the above results with the ones obtained on the period domain.

\begin{prop} Let $\theta(\xi) = u$.  then the action of
$R_D(\xi,\bar\xi)$ on $\UU^p$ agrees with the action of
$R_{\nabla}(u,\bar u)$ on $U^p$ from Corollary \eqref{corr:hodge-bundle}.
More precisely,  the four terms in the expression for $R_{\nabla}(u,\bar u)$
compare as follows
\[
\aligned
{[}\Pi'\comp (\bar u^*_+),
                                  \Pi'\comp (\bar u_+)]  &=  (\theta\wedge\tau_0 + \tau_0\wedge\theta)(\xi, \bar \xi) \\
-  [u,\bar u]_0 &= -( \theta_0\wedge\tau_0 + \tau_0\wedge\theta_0)(\xi,\bar\xi)\\
-\Pi'\comp [u,\bar u]_+&=   -(\theta_-\wedge\tau_0 + \tau_0\wedge\theta_-)(\xi,\bar\xi) ,\\
-\Pi'\comp [u,\bar u]_+^*&=   -(\theta_-\wedge\tau_0 + \tau_0\wedge\theta_-)^*(\xi,\bar\xi).
\endaligned
\]
\end{prop}
\proof
Recall that
for vector valued $A$ of type $(1,0)$ and $B$ of type $(0,1)$ we have
$$
          (A\wedge B + B\wedge A)(\xi,\bar\xi) = [A(\xi),B(\bar\xi)].
$$
A check of Hodge types shows that $\tau_-(\xi) = \Pi'\circ (\bar u)_+$ and hence
$$
           -(\tau_-\wedge\tau_-^* + \tau_-^*\wedge\tau_-)(\xi,\bar\xi)
           = -[\Pi'\circ (\bar u)_+^*,\Pi'\circ (\bar u)_+]
$$
which is the first term of $R_\nabla(u,\bar u)$.  The partial term
$$
        -(\theta_0\wedge\tau_0 + \tau_0\wedge\theta_0)(\xi,\bar\xi)
        = -[u,\bar u]_0
$$
is extracted from $-(\theta\wedge\tau_0 + \tau_0\wedge\theta)$.  What
remains of this term,
$$
          -(\theta_-\wedge\tau_0 + \tau_0\wedge\theta_-),
$$
computes $-\Pi'\circ[u,\bar u]_+$.
\qed\endproof


\section{Special Case:   $W_{-1}\eg_\C$ is Abelian} \label{defoinAb}
\subsection*{Negative Curvature}
{%
Consider a period map
\[
F:\Delta \to D,\quad s\mapsto F(s).
\]
One lets    $\pi_\eq^{F(s)}$ denote  projection onto $\eq_{F(s)}$ via the decomposition
$$
        \eg_{\C} = \eg_{\C}^{F(s)}\oplus \eq_{F(s)}.
$$
By Lemma~ \ref{TangentIdent}  we have 
\[
F(s)= e^{\Gamma(s)}\cdot F(0),
\]
\changed{}{where $\Gamma: \Delta\to   \eq_{F(0)}$ is a holomorphic arc.}
The following   expression for the pushforward
vector  field $d/ds$ on $\Delta$ is needed below:
%
\begin{lemma}  \label{pushforwardtangents} We have
\begin{equation}
          F_*\left(\frac{d}{ds}\right)
          = \pi_\eq^{F(s)}
                \psi_1\left (\Gamma(s), \left( \frac{d\Gamma}{d s} \right)\right),
                                                        \label{eqn:2.11}
                                                        \end{equation}
where we recall \eqref{eqn:2.10} that $\displaystyle  \psi_1(u,v) = \frac{e^{\ad
{u}}-1}{\ad {u} - 1}v$.
\end{lemma}
\proof
\changed{}{We have}
\[
\aligned
F(s)&= e^{\Gamma(s)}  e ^{-\Gamma(p)} F(p)\\
      &= e^{\Gamma(p)+[\Gamma(s)-\Gamma(p)]}  e ^{-\Gamma(p)} F(p). \\
\endaligned
\]
The Campbell-Baker-Hausdorff formalism \eqref{eqn:2.10}  shows that
\[
e^{\Gamma(p)+[\Gamma(s)-\Gamma(p)] } e^{-\Gamma(p)}= e^{\psi_1(\Gamma(p), \Gamma(s)-\Gamma(p)) }.
\]
Since  $\Gamma(s)-\Gamma(p) =   (s-p)  \frac{d\Gamma}{d s}(p) +O((s-p)^2)$,
we have
\[
e^{\psi_1(\Gamma(p), \Gamma(s)-\Gamma(p) }  =     e^{\psi_1(\Gamma(p), \frac{d\Gamma}{d s}(p))(s-p)  + O((s-p)^2)}.
\]
So, for a given  test function $\zeta$ at $F(p)$, we have
\begin{equation*}
        \aligned
        F_*\left(\frac{d}{ds}\right)_p\zeta &= \left(\frac{d}{ds}\right)_p\zeta(e^{\Gamma(s)} \cdot F(0)) \\
        &=   \left(\frac{d}{ds}\right)_p    \zeta (   e^{ (s-p)\psi_1(\Gamma(p), \frac{d\Gamma}{d s}(p))}   \cdot F(p)).\\
                               \endaligned
\end{equation*}
The result then follows applying again Lemma~\ref{TangentIdent} but now for the identification of $T_{F(p)}D$ and $\eq^{F(p)}$ (in loc. cit. take $t=s-p$
and $u=   \frac{d\Gamma}{d s}(p)$).
\qed\endproof}

\begin{prop} \label{negcurv} Let
\[
F:\Delta \to D,\quad s\mapsto F(s),
\]
be the period map of a unipotent variation of mixed Hodge structure (i.e.
the induced variations on $\gr^W$ are constant) and suppose further that
$W_{-1}\eg_{\C}$ is abelian.  Then the holomorphic sectional curvature of the
pull back metric is $\le 0$.
\end{prop}
\proof We have seen  in Corollary~\ref{curvapll2} that the holomorphic sectional curvature of the Hodge  metric on $D$ at $F(0)$ is semi-positive.  However, when  we pull back a metric, the curvature gets an extra term which is $\le 0$. 
We   shall  show  that due to the fact that   $W_{-1}\eg_{\C}$
is abelian, the pull back metric gains sufficient
negativity to compensate positivity. \par
By the choice of coordinates
\eqref{eqn:coord}, we can write the  period map in the local normal form
$$
           F(s) = e^{\Gamma(s)} \cdot F(0),
$$
where $\Gamma(s)$ is a holomorphic function taking values in the
intersection of $W_{-1}\eg_{\C}$ and  $\eq = \eq_{F(0)}$, i.e. $\Gamma(s)\in \eg^{-1,0} \oplus\eg^{-1,1}$.   
{Then $\overline{\Gamma(s)}  \in \eg^{0,-1}+\Lambda$ and Kaplan's decomposition
(Theorem~\ref{group-decomp})
 in this situation simplifies to
\begin{equation}\label{eqn:w2AbelianSimplifies}
        e^{\Gamma(s)} =  \underbrace{e^{\Gamma(s) + \bar\Gamma(s)}}_{g_{\R}(s) }
      \cdot\underbrace{ e^{-\pi_\Lambda (\bar\Gamma(s))} }_{e^{\lambda(s)}} \cdot  \underbrace{e^{-\pi_+(\bar\Gamma(s))}}_{f(s)}
    \end{equation}
thanks to the fact that $W_{-1}\eg_{\C}$ is abelian.}

{The relation  \eqref{eqn:2.11}  becomes
\begin{equation}\label{eqn:2.11bis}
 F_*\left(\frac{d}{ds}\right)
          = \pi_\eq^{F(s)}\left(
                  \frac{d\Gamma}{d s}\right),
\end{equation}
since  $\psi_1(\Gamma(p), \frac{d\Gamma}{d s}(p))= \frac{d\Gamma}{d s}(p)$: indeed, in our case $\Gamma(p)$ and $\frac{d\Gamma}{d s}(p)$ commute.
Next  we need to replace  $\pi_\eq^{F(s)}$ by an expression involving
$\pi_\eq=\pi_\eq^{F(0)}$ since we want to calculate the Hodge metric at $F(0)$.  Now  note that  $\Ad{g_\R(s)\cdot e^{\lambda(s)}}$ maps $ \End (V)^{i,j}_{F(0)}$  to $\End (V)^{i,j}_{F(s)}$ and since $\pi_\eq^{F(s)}$ is defined in terms of projections onto such components,
\begin{equation*}
\aligned
 \pi_\eq ^{F(s)}& = \Ad{g_\R(s)} \cdot  \Ad{e^{\lambda(s)} }\comp   \pi_\eq \comp   \Ad{e^{-\lambda(s)} } \cdot
 \Ad{g^{-1}_\R}(s) \\
                       &=  \Ad{g_\R(s)} \cdot  \Ad{e^{\lambda(s)} }\comp \pi_\eq \comp \Ad{e^{\varphi(s)} } \cdot \Ad{e^{-\Gamma(s)}}.
\endaligned
\end{equation*}
Remark that \eqref{eqn:w2AbelianSimplifies} shows that $\varphi(s)=-\pi_+ (\bar \Gamma(s))\in\eg^{0,-1}$.  Using all of this, again by commutativity,  \eqref{eqn:2.11bis} becomes
\begin{equation}\label{eqn:5.11}
F_*\left(\frac{d}{ds}\right)
          =   \Ad{g_\R(s)} \cdot  \Ad{e^{\lambda(s)} } \left(\frac{d \Gamma}{ds}\right) .
\end{equation}
Note that $\Ad{g_\R(s)} \cdot  \Ad{e^{\lambda(s)} }$ acts by isometries and so 
$$
\aligned
       h(s)  &:=  \left \| (  F_*\left(\frac{d}{ds}\right)  \right \|_{F(s)} \\
       &=\left \|   \left(\frac{d \Gamma}{ds}\right) \right \|_{F(0)}.
\endaligned
$$ }
 \par
The function $\xi(s) = \displaystyle \frac{d \Gamma}{ds}$ is a holomorphic function and so $\displaystyle \frac{\del \xi(s)}{\del\bar s}=0$.
Put $\dot\xi=\displaystyle \frac{d \xi(s)}{d s}$ and $h_o = h_{F(0)}$.  Then, the curvature of the  pullback metric is:
$$
\aligned
        K &= -\frac{1}{h}\frac{\pd^2}{\pd s \pd \bar s}\log h
           = -\frac{1}{h_o(\xi,\xi)}\frac{\pd^2}{\pd s \pd \bar s}\log h_o(\xi,\xi) \\
          &= -\frac{1}{h_o(\xi,\xi)}\frac{\pd}{\pd s}\left(
                \frac{h_o(\xi,\dot\xi)}{h_o(\xi,\xi)}\right)          \\
          &= -\frac{1}{h_o(\xi,\xi)}
                \frac{h_o(\dot\xi,\dot\xi)h_o(\xi,\xi) -
                   h_o(\dot\xi,\xi) h_o(\xi,\dot\xi)}
                {h_o(\xi,\xi)^2}                                              \\
          &= \frac{| h_o(\dot\xi,\xi)|^2 -
                   h_o(\dot\xi,\dot\xi)h_o(\xi,\xi)}
                {h_o^3(\xi,\xi)} \leq 0,
\endaligned
$$
where the last step follows  from the Cauchy-Schwarz inequality for
$h_o(\dot\xi,\xi)$.  \qed\endproof

\begin{rmq}   The proof shows that  the Gaussian curvature of
the pullback is negative wherever $\xi$ and $\dot\xi$ are linearly
independent.
\end{rmq}

\medskip

\par In particular, Proposition \eqref{negcurv} yields:
\begin{corr} \label{curvneg:2}  Let $\Delta\to D$ be a period map  associated to a   normal function
with fixed underlying Hodge structure.  Then the holomorphic sectional
curvature of the pull back of the Hodge metric is semi-negative.
\end{corr}

\begin{rmk} Via isomorphism
$\Ext^1_{\text{\rm MHS}}(A,B)\cong\Ext^1(\Z(0),B\otimes A^{\vee})$, the observation
of the previous paragraph also applies to families of cycles on a fixed
variety $X$ and the VMHS on $J_x/J_x^3$ of a smooth projective variety.
\end{rmk}

\subsection*{Another Application:   Mixed Hodge Structures and Fundamental Groups}
We treat this in some detail  with an eye towards
a reader less acquainted with this material. \par
 Let $X$ be a smooth complex algebraic variety, and $\Z\pi_1(X,x)$
be the group ring consisting of all finite, formal $\Z$-linear
combinations of elements of $\pi_1(X,x)$.  The augmentation
ideal $J_x$ is defined to be the kernel of the ring homomorphism
$$
      \epsilon:\Z\pi_1(X,x)\to\Z
$$
which maps each element $g\in\pi_1(X,x)$ to $1\in\Z$.
By the work of Morgan \cite{M}, the quotients $J_x/J_x^k$ carry
functorial mixed Hodge structures constructed from the minimal
model of the de Rham algebra of $X$.
We follow Hain's alternative approach  \cite{hain}; the mixed Hodge structure on $J_x/J_x^k$
can be described using so called iterated integrals as follows:
The iterated integral on   $\theta_1,\dots,\theta_r\in\EE^1(X)$,
$$
        \int\, \theta_1\cdots\theta_r
$$
assigns to each smooth path $\gamma:[0,1]\to X$ the integral of
$\theta_1\cdots\theta_r$ over the standard simplex in $\R^r$,
i.e.
$$
        \int_{\gamma}\, \theta_1\cdots\theta_r
        = \int_{0\leq t_1\leq\cdots\leq t_r\leq 1}\,
          \theta_1(\gamma_*(d/d t_1))\cdots \theta_r(\gamma_*(d/d t_r))
	  dt_1\cdots d t_r.
$$
Such an iterated  integral is said to have length $r$. The spaces $\Hom_{\Z}(J_x/J^{s+1}_x,\C)$ can be described as
spaces of certain linear combinations of iterated integrals of lengths $\le s$, the so called \emph{homotopy functionals}.
We only need their description for $s=2$:
 \begin{thm}[\protect{ \cite[Prop. 3.1.]{hain}}]\label{thm:it-int} The iterated integral
\begin{equation}
       \int\,\theta + \sum_{j,k}\, a_{jk}\int\,\theta_j\theta_k
       \label{eqn:it-int}
\end{equation}
is a homotopy functional if and only if $\theta_1,\dots,\theta_r$ are
closed and
\begin{equation}
      d\theta + \sum_{jk}\, a_{jk}\theta_j \wedge \theta_k = 0.
      \label{eqn:it-int-eq}
\end{equation}
\end{thm}
The mixed Hodge structure  $(F,W)$ on $\Hom_{\Z}(J_x/J^{s+1}_x,\C)$
is described on the level of iterated integrals as follows.
Such a sum belongs $F^p$ if and only if  each integrand $\theta_1\cdots\theta_k$
contains at least $p$ terms $\theta_j\in\Omega^1(X)$.  As for the
weight filtration, $\alpha$ belongs to $W_k$ if and only if $\alpha$ is
representable by a sum of iterated integrals of length $\leq k$
 plus the number of logarithmic terms $dz_j/z_j$ in the integrand.

Suppose next that $H^1(X)$ has pure weight $\ell=1$ or $\ell=2$. The first happens
for $X$ projective, the second for instance when the compactification of $X$ is
$\bP^1$. In these situations,  following \cite[\S 6]{hain},  the dual of $J_x/J_x^3$ is an extension of pure Hodge structures.
To explain the result,  note that the cup-product pairing
$H^1(X) \otimes H^1(X)\to H^2(X)$ is a morphism of pure  Hodge structures. It follows that
\[
K:=\ker \left[ H^1(X) \otimes H^1(X)\to H^2(X)\right]
\]
carries a pure  Hodge structure  of weight  $2\ell$.  
Theorem~\ref{thm:it-int} now implies:
\begin{thm}  The mixed Hodge structure on $\Hom_\Z(J/J^3,\C)$ is the extension of pure Hodge structures of weight $\ell$ and $2\ell$ given by
\[
0\to H^1(X) \to \Hom_\Z(J/J^3,\C)\mapright{p} K \to 0.
\]
Explicitly, the iterated integral $  \int\,\theta + \sum_{j,k}\, a_{jk}\int\,\theta_j \theta_k$ is mapped   by  $p$ to $\sum a_{jk} [\theta_j]\otimes [\theta_k]$ which, by construction, belongs to $K$. The kernel of $p$ can be identified with the the length one homotopy  integrals  $ \int\,\theta$ , i.e. those with  $d\theta =0$.
Hence $\ker p\simeq H^1(X)$. It follows that the graded pieces have a natural polarization coming from the one on $H^1(X)$ and which is  given by these identifications.
\end{thm}
In particular, the above implies that if $X$ is smooth projective, the graded polarized mixed Hodge structure on $\Hom_\Z(J/J^3,\C)$ has two adjacent weights and so if we now leave $X$ fixed but vary the base point, we get a family  of  mixed Hodge structures over $X$ for  which $W_{-1}\eg_\C$ is abelian and by  Proposition~\ref{negcurv} we conclude:
\begin{corr}\label{corr:fund-group} Let $X$ be a smooth complex projective
variety, and suppose that the differential of the period map of $J_x/J_x^3$
is injective. Then the holomorphic sectional curvature of $X$ is  $\le 0$.
\end{corr}

\subsubsection*{Complements: Flat Structure and  the Hodge Metric}
1.  The  \textit{flat structure} given by the local system attached to $J/J^3$ may be described
as follows:  Fix a point $x_o\in X$ and let $U$ be a simply connected
open subset containing $x_o$.  Given a point $x\in U$ let
$\gamma:[0,1]\to U$ be a smooth path connecting $x_o$ to $x$.  Then,
conjugation
\begin{equation}
      \alpha\mapsto\gamma \alpha \gamma^{-1}    \label{eqn:conj-pi1} 
\end{equation}
defines an isomorphism $\pi_1(X,x)\to\pi_1(X,x_0)$ which is
independent of $\gamma$ since $U$ is simply connected.  Trivializing
$(J/J^3)^*$ using \eqref{eqn:conj-pi1}, we then obtain the period map via
the change of base point formula  (see \cite[Remark 6.6]{hain}):
\begin{equation}
	\int_{\gamma\alpha\gamma^{-1}}\,\theta_1\theta_2
	 = \int_{\alpha}\,\theta_1\theta_2
   + \left(\int_{\gamma}\,\theta_1\right)\left(\int_{\alpha}\,\theta_2\right)
   - \left(\int_{\gamma}\,\theta_2\right)\left(\int_{\alpha}\,\theta_1\right)
	  \label{eqn:change-base-pt}
\end{equation}
one then obtains the following result via differentiation:
\begin{lemma}\label{lemma:gauss-manin} The flat connection $\nabla$ of
$(J/J^3)^*$ operates on iterated integrals via the following rules:
$$
	  \nabla_{\xi}\left(\int\,\theta_1\theta_2\right)
	  =\theta_1(\xi)\left(\int\,\theta_2\right)
	   - \theta_2(\xi)\left(\int\,\theta_1\right)
$$
and $\nabla_{\xi}(\int\,\theta) = 0$.
\end{lemma}
As a check of the formula for $\nabla$ given in Lemma~\ref{lemma:gauss-manin},
note that by Theorem \ref{thm:it-int} that the iterated integral
\eqref{eqn:it-int}
appears in $(J/J^3)^*$ only if $\theta_j$ and $\theta_k$ is closed
for all $j$, $k$ an equation \eqref{eqn:it-int-eq} holds.
Therefore,
$$
     \nabla^2
     \left(\int\,\theta
		+ \sum_{j,k}\, a_{jk}\int\,\theta_j \theta_k\right)
      = \sum a_{ij}\left(d\theta_j\int\theta_k - d\theta_j \int\theta_k\right) = 0
$$
because $d\theta_j=0$.
Likewise, direct calculation using Lemma \eqref{lemma:gauss-manin} shows that
the Hodge filtration $\FF$ of $(J/J^3)^*$ is holomorphic and horizontal
with respect to $\nabla$, and the weight filtration $W$ is flat.

\noindent 2. By way of  illustration we shall prove  the correctness of the expression~\eqref{eqn:dilog-ex} for
\textit{the mixed Hodge metric} as announced in the introduction.
First of all (for $X=\bP^1\setminus \set{0,1,\infty}$)
	$$
	    \nabla\int \frac{dz}{z} \cdot \frac{dz}{1-z} = \frac{dz}{z}\int\frac{dz}{z-1}
	                                      -\frac{dz}{z-1}\int\frac{dz}{z},
	$$
and, secondly,  from the above discussion  it follows that
	\[
	\left\| \int\frac{dz}{z-1}\right \|^2= h([\frac{dz}{z-1}] ,[\frac{dz}{z-1}])= (4\pi)^2.
	\]
where $h$ is the Hodge metric  on $H^1(X)$ (and similarly for $\|\int \frac{dz}{z} \|^2$).

\noindent 3. As a further illustration, let us calculate \textit{the mixed Hodge metric}   when we
specialize the preceding to a  compact Riemann surface   $X$ of genus
$g>1$. Let $\theta_1,\dots,\theta_g$ be an unitary  basis of
$H^{1,0}(X)$ with respect to  the Hodge metric. Then, up to a scalar, the
metric on $X$ obtained by pulling back the mixed Hodge metric via
the period map of $(J/J^3)^*$ is given by
$$
         \left\|d/dz\right\|^2 = \sum_{j=1}^g\, \|\theta_j(d/dz)\|^2.
$$
This follows directly from Lemma \eqref{lemma:gauss-manin} and the
discussion on the mixed Hodge structure on $(J/J^3)^*$ we just gave.
\begin{rmq} The above description of the mixed Hodge metric  can
be generalized in a straightforward manner  to  any smooth complex projective variety.
\end{rmq}
%
%
%
%

\section{The K\"ahler Condition} \label{sec:kaehler}

{We recall some facts about K\"ahler metrics.
Let $h$ be a hermitian
metric on a complex manifold $M$.   Given any system of local holomorphic
coordinates $(z_1,\dots,z_m)$ on $M$, the associated fundamental 2-form
$\Omega$ is given by the formula
\begin{equation} \label{eqn:kahler}
	\Omega = -\frac{\sqrt{-1}}{2}\sum_{j,k} h_{jk} dz_j\wedge d\bar z_k,
	\qquad h_{jk} = h\left(\frac{\del}{\del{z_j}}, \frac{\del}{\del{z_k}}\right).
\end{equation}
This form is a globally defined $(1,1)$-form and  by definition $h$ is K\"ahler
  if and only if  $d\Omega = 0$.
}

{An equivalent condition can be given in terms of the  torsion tensor for the associated Chern connection $\nabla_h$ on the holomorphic tangent bundle. Recall that   the  \emph{torsion tensor} for any  linear connection $\nabla$ on the tangent bundle  is defined by the formula
\[
T_\nabla(X,Y):= \nabla_XY -\nabla_YX -[X,Y],
\]
where $X$ and $Y$ a local smooth vector fields.
The K\"ahler condition is  equivalent to $T_{\nabla_h}=0$. see~\cite[Prop. I.7.19]{koba}.
 \begin{prop} \label{torsiontest}
A  hermitian metric $h$ as above  with Chern connection $\nabla=\nabla_h$   is K\"ahler if and only if
for  local holomorphic vector fields   $X,Y $ on  $M$ one has
\[
\nabla_{X}  Y  -
  \nabla_{ Y} X - [X,Y] =0.
\]
\end{prop}
\proof  The torsion is a tensor, i.e. bilinear over $C^\infty(M)$ and since all local vector fields are $C^\infty(M)$-linear combinations of the holomorphic coordinate vector fields and their complex conjugates, it suffices to test  whether $T(X,Y)=0$ with $X$ and $Y$   locally  holomorphic or anti-holomorphic.
If $X$ and $Y$ have different types one has $[X,Y]=0$ \footnote{\label{holandantiholcommute} Clearly, if $X$, $Y$ are local holomorphic coordinate vector fields $[X,\bar Y]=0$ and an easy calculation shows that $[fX, \bar g \bar Y]=0$  whenever  $f,g$ are local holomorphic functions and $X, Y$ holomorphic fields with $[X,\bar Y]=0$.}
and hence the torsion vanishes on such pairs $(X,Y)$. Since $T(\bar X,\bar Y)= \overline{T(X,Y)}$, to show that the torsion vanishes, one therefore may restrict to pairs $(X,Y)$ of local holomorphic vector fields.  So  $T=0$ precisely if $T$ vanishes on pairs of vector fields belonging to a holomorphic local frame  for the holomorphic tangent bundle.
\qed\endproof
}

\par
Let $\Delta^m$ a  polydisk  at $0\in \C^m$ with  coordinates $(s_1,\dots,s_m)$  and let
 $F:\Delta^m\to D$ be a holomorphic, horizontal map.  Let $\eq$ be the subalgebra
\eqref{eqn:tang-alg} attached to $F(0)$.  Recalling the local
biholomorphism   \eqref{eqn:coord}  mapping a neighborhood of $0\in\eq$ to a
neighborhood of $F(0)$ in $D$, locally we can write as in ~\cite{higgs}
$$
      F(s) = e^{\Gamma(s)}  \cdot F(0)
$$
for a unique $\eq$-valued \emph{holomorphic} function $\Gamma$ which vanishes at $0$.

\begin{thm}\label{thm:kahler} Let $h = F^*(h_D)$ denote the pullback
of the mixed Hodge metric $h_D$ to $S$.
Set $\xi_j = \frac{\partial\Gamma}{\partial s_j}(0)$. Then $h$
is K\"ahler if and only if for all $j,k,\ell$ one has
\begin{equation}
           h(\xi_j,\pi_\eq [\pip(\bar\xi_{\ell}),\xi_k])
           - h(\xi_{\ell},\pi_\eq[\pip(\bar\xi_j),\xi_k]) = 0. \label{eqn:1.3bis}
\end{equation}
\end{thm}
\proof
%
%
 {First, remark that  by Theorem~\ref{thm:connection} one has
 \[
 \nabla_{\xi_j} \xi_\ell= - \pi_\eq[\pi_+(\overline{\xi_j})^*,\xi_\ell] . \]
Since
\[
\aligned h(  \pi_\eq[\pi_+(\overline{\xi_j})^*,\xi_\ell], \xi_k) &=  h([\pi_+(\overline{\xi_j})^*,\xi_\ell], \xi_k)\\
&= h(\xi_\ell, [\pi_+(\overline{\xi_j}), \xi_k])\\\
&=  h(\xi_\ell, \pi_\eq[\pi_+(\overline{\xi_j}), \xi_k]) ,
\endaligned
\]
formula \eqref{eqn:1.3bis} for all $j,k,\ell$  is equivalent to
\[
\nabla_{\xi_j} \xi_\ell- \nabla_{\xi_\ell} \xi_j=0 \quad \text{for all } \ell, j
\]
and hence, by  the second condition from Prop.~\ref{torsiontest} we only have to show that  the bracket $[\xi_j,\xi_\ell]$ vanishes.
}

{To see this, recall that  period maps are \emph{horizontal}, i.e.
all tangents to the image $F(s)$ of a period map belong to $U^{-1}_{F(s)}=\bigoplus_q I^{-1,q}_{F(s)}$. Working this out means
\[
e^{-\ad{\Gamma(s)}} \frac{\del}{\del s_j} e^{\ad{\Gamma(s)}} \in U^{-1}_{F(0)}
\]
and as in the proof of \cite[Theorem 6.9]{higgs}
this is equivalent to the commutativity relation
\begin{equation*}  
        [\xi_j,\xi_\ell]=   \left[\frac{\del \Gamma}{\del s_j}(0),\frac{\del \Gamma}{\del s_\ell}(0)\right]
           = 0.\quad\qed
\end{equation*}
}

 \begin{corr}\label{corr:kahler} The pullback of the mixed Hodge metric
along an immersion is K\"ahler in the following cases:
\renewcommand{\labelenumi}{\rm (\alph{enumi})}
\begin{enumerate}
\item Variations of pure Hodge structure (Lu's result~\cite{Zhiqin});
\item  Hodge--Tate variations;
\item The variations of mixed Hodge structure attached to
$J_x/J_x^3$ for a smooth complex projective variety;
\item The variations from \S~\ref{ssec:exmples}. Example 4 arising from the commuting deformations of
the complex and K\"ahler structure of a compact K\"ahler manifold.
\end{enumerate}
\end{corr}
\begin{proof} In case (a), the derivatives of $\Gamma$ at zero are of
type $(-1,1)$ and so  for all $\ell,j$
\begin{equation}
   [\pip(\overline{\xi_\ell }),\xi_j]= [\pip(\overline{d\Gamma/ds_{\ell}(0)}),d\Gamma/ds_j(0)]  \label{eqn:brac-1}
\end{equation}
is type $(0,0)$ which is annihilated by $\piq$.
\\
 In case  (b),
$\pip(\overline{d\Gamma}) = 0$.
\\
 In case (c) the bracket
\eqref{eqn:brac-1} is of type $(-1,-1)$ which is zero due to the
short length of the weight filtration.  \\
In case (d), the bracket
\eqref{eqn:brac-1} has terms of type $(0,0)$ and $(0,-2)$, both of
which are annihilated by $\pi_\eq$.
\end{proof}
\begin{rmk} In case (d) one can also show that the the holomorphic sectional
curvature will be $\leq 0$.
\end{rmk}

\begin{thm}\label{thm:kahler2} Let $\VV$ be a variation of mixed
Hodge structure with only two non-trivial weight graded-quotients
$\gr^W_a$ and $\gr^W_b$ which are adjacent, i.e. $|a-b|=1$.  Then,
the pullback of the mixed Hodge metric along the period map of
$\VV$ is \changed{}{a K\"ahler pseudometric}.
\end{thm}
\proof
We shall prove the symmetry relation \eqref{eqn:1.3bis} which in
our situation due to the
 short nature of the weight filtration reduces to
\begin{equation}
           h(\xi_j,[\bar\xi_{\ell},\xi_k])
           - h(\xi_{\ell},[\bar\xi_j,\xi_k]) = 0.      \label{eqn:kahler-1}
\end{equation}
Without loss of generality, we can assume that $\xi_j$, $\xi_k$, $\xi_{\ell}$
are of pure Hodge type.  Inspection of the possibilities shows that the
only non-trivial case is when $X=\xi_j$ and $Y=\xi_{\ell}$ are type $(-1,0)$
and $Z=\xi_k$ is type $(-1,1)$.  Since by  Lemma~\ref{lemma:adjoint} we have
$Z^*=-\bar Z$ in this case, the formula \eqref{eqn:MetConv} and the fact that
$h$ is hermitian gives
$$
\aligned
h(X,[\bar Y, Z])&=h([X,\bar Z], \bar Y))\\
&=h(Y, [\bar X,Z]),
\endaligned
$$
which is \eqref{eqn:kahler-1}.\qed\endproof
%
%

\begin{exmple} \label{kaehlerexmples} In particular, Theorem~\ref{thm:kahler2} applies to the
tautological variations of Hodge structure over the moduli spaces
$\mathcal M_{g,n}$ and more generally, to  families of pairs $(X_s,Y_s)$ of a
smooth projective variety $X_s$ and a smooth hypersurface $Y_s\subset X_s$
as well as a family of normal functions \eqref{eqn:nf-2} over  a curve  $S$ with $\mathcal H$ fixed and whose period map is an immersion.
\end{exmple}

\section{The Biextension Line Bundle}\label{sec:biexts}

Recall from the introduction that in this special case for  the graded Hodge numbers  we have $h^{-1,-1}=1$ and all other $h^{p,q}=0$ unless $p+q=-1$; the mixed Hodge structure is described as a biextension
\begin{equation}\label{eqn:biextension}
\aligned
 0 \to & \gr^W_{-1}  \to W_0/W_{-2} \to \gr^W_0=\Z(0) \to 0\\
   0 \to  & \gr^W_{-2} =\Z(1) \to  W_{-1} \to       \gr^W_{-1}  \to 0.
\endaligned
\end{equation}
 As explained below, a  family of such mixed Hodge structures over a parameter space $S$ comes with a biextension metric $h_{\rm biext}(s)$.  Its  Chern form   will be shown  to be  semi-positive along any curve, provided the biextension is self-dual: see Theorem.~\ref{plurisubharmonicity}.
\par
The point in this section
is that the mixed Hodge structure is in general \emph{not} split and that the metric $h_{\rm biext}$ can be found
by  comparing  the given mixed Hodge structure $(F,W)$ on the real vector space $W_0$ to its Deligne splitting
$(e^{-\ii \delta_{F,W}} F,W)$ where  we recall from \cite[Prop. 2.20]{degeneration}
that
\begin{equation} \label{eqn:delta}
\delta_{F,W}=\half\im Y _{F,W} = \frac{1}{4\ii} (Y_{F,W} - \bar Y_{F,W})  \in  \Lambda_{F,W}\cap \eg_\R.
\end{equation}
Here  $Y_{F,W}\in \End(V_\C)$ equals  multiplication by $p+q$
on Deligne's $I^{p,q}(V)$.
\par
Since $\gr_{-2}^W  \simeq \R$ and similarly for $\gr_0^W $, fixing bases, the map $\delta_{F,W}$ can
then be viewed as a real  number $\delta$, depending on $(F,W)$. By \cite[\S 5]{nilp}, there exists a further real number $\lambda$ depending only on $W$ such that the positive number
$
 h(F,W)= e^{-2\pi \delta/\lambda}
$
depends only on the equivalence class of the extension.
\par
Let us apply this in our setting of a family $(\FF,W)$ of biextensions over a complex curve $S$. Then
\begin{equation}\label{eqn:biextmetr}
h_{\rm biext}(s):= h(F_s,W)=  e^{ -2\pi \delta(s)},\quad \delta(s)= \frac{ \delta_{F_s,W}}{\lambda}
\end{equation}
turns out to be a hermitian metric on $S$.  
\par
As before  we write
\begin{equation}
         F(s) = e^{\Gamma(s)}\cdot F,                 \label{eqn:lnf}
\end{equation}
where $F=F(0)$ and $\Gamma(s)$ is a holomorphic function on a coordinate patch  in $S$
with values in $\eq$.  This is the main result we are after:

\begin{thm} \label{chernformofhm} Let $S$ be a curve and let  $\mathcal \FF $ be a variation of biextension type over $S$ with
local normal form \eqref{eqn:lnf}.  Let $\gamma^{-1,0}$ be  the Hodge component of type $(-1,0)$ of $\Gamma'(0)$.

The Chern form of the biextension metric \eqref{eqn:biextmetr}  is the $(1,1)$--form
\begin{equation}\label{eqn:biextrel}
\aligned
\frac{1}{2\pi\ii} \del\bar\del  \log
\left( h_{\rm biext}(s)\right) = & \ii   \frac{\del ^2  \,\,\delta(s)}{\del s \del \bar s \phantom{X} }   ds\wedge \overline{ds}\\
 = &\half  [\gamma^{-1,0},\bar\gamma^{-1,0}] \,ds \wedge \overline{ds}.
\endaligned
\end{equation}
\end{thm}
\proof
 Let
\begin{equation}
        e^{\Gamma(s)} = g_{\R}(s) e^{\lambda(s)} f(s) \label{eqn:group-decomp}
\end{equation}
as usual.  Then, by Lemma~\ref{higgs} we have $Y(s)= g_\R(s)e^{\lambda(s)} Y$, where $Y = Y_{(F,W)}$.
If we set  $f(s) = e^{\varphi(s)}$, using \eqref{eqn:delta},  we get
\begin{equation}
       \frac{\del^2   }{\del \bar s \del s}  \delta(s)=
        \frac{1}{2}\im
       \frac{\del^2}{\del s \del\bar s}
          \underbrace{e^{\Gamma(s)}e^{-\varphi(s)}}_{d(s)}\cdot Y .              \label{eqn:delta-laplacian}
\end{equation}
Since $\Gamma(s)$ is holomorphic, we have
$$
      \frac{\del}{\pd \bar s} d(s) \cdot Y
        =\Ad(e^{\Gamma(s)})\left(\frac{\pd}{\pd\bar s} e^{- \ad{\varphi(s)} }\cdot  Y\right)
$$
and so
\begin{equation}
\aligned
       \frac{\del^2   }{\del  s \del  \bar s} d(s) \cdot  Y
      =&  \left(\frac{\pd}{\pd s} e^{\ad\Gamma(s)}\right)
          \left(\frac{\pd}{\pd\bar s}e^{-\ad \varphi(s)} Y\right) \\
          & \, +
           \Ad e^{\Gamma(s)}
          \left(\frac{\del^2 e^{-\ad \varphi(s)} }{\del s \del \bar s}
                 Y \right).    \label{eqn:laplace-1}
\endaligned
\end{equation}
{We now consider the Taylor expansion (note that $\varphi(0)=0$)
$$
       \varphi(s) =  \varphi_{01}s +\varphi_{10}\bar s+ \sum_{j,k}\, \varphi_{jk}\, s^j \bar s^k+ O^3(s,\bar s).
$$
}By Lemma~\ref{secondorder}, we also know
\begin{eqnarray} \varphi_{10}&=&0, \label{eqn:expand1}\\
 \varphi_{01}& =&  - (\overline{\Gamma'(0)})_+, \label{eqn:expand2}\\
  \varphi_{11}& =  & [\gamma,\bar \gamma] _0+ [\gamma,\bar \gamma] _+ \nonumber\\
                 &= &   [\gamma^{-1,1},\bar \gamma^{-1,1}] _0 +  [\gamma^{-1,1},\bar \gamma^{-1,0}]. \label{eqn:expand3}
\end{eqnarray}
Formula \eqref{eqn:expand1} shows
 that  the term with $s\bar s$ in the Taylor expansion of
 $$\frac{\del^2}{\del s \del \bar s} e^{-\ad \varphi(s)} Y$$
is just  $-[\varphi_{11},Y]$.  Together with equation \eqref{eqn:laplace-1}
it follows that
\begin{equation}\label{eqn:dbardgamma}
    \left.\frac{\del^2}{\del s \del\bar s}d(s)\cdot Y
     \right|_0
     =- [\Gamma'(0),[\varphi_{01},Y]] - [\varphi_{11},Y]
\end{equation}
Eqn.~\eqref{eqn:expand2} states that $\varphi_{0,1}= -\overline{\Gamma'(0)}_+$.
Let $\gamma = \Gamma'(0)$.  By horizontality and the short length of
the weight filtration,
$$
       \gamma = \gamma^{-1,1} + \gamma^{-1,0} + \gamma^{-1,-1}.
$$
Moreover, since $(F,W)$ is a biextension
\begin{equation*}
       \bar \gamma^{-1,1} \in\mathfrak g^{1,-1},\qquad
       \bar \gamma^{-1,0} \in\mathfrak g^{0,-1},\qquad
       \bar \gamma^{-1,-1} \in\mathfrak g^{-1,-1}        
\end{equation*}
Therefore,
\begin{equation*}
      - \varphi_{01} =  (\overline{\Gamma'(0)})_+
                 = \bar\gamma^{-1,1} + \bar\gamma^{-1,0}.
\end{equation*}
In particular, since $\ad Y$ acts as multiplication by $a+b$ on
$\mathfrak g^{a,b}$ it follows that
\begin{equation} \label{eqn:phi-01}
\aligned
        -  [\Gamma'(0),[\varphi_{01},Y]]
          &=& [\gamma,[\bar\gamma^{-1,1} + \bar\gamma^{-1,0},Y]]
           = [\gamma,\bar\gamma^{-1,0}] \\
          &=& [\gamma^{-1,1},\bar\gamma^{-1,0}] + [\gamma^{-1,0},\bar\gamma^{-1,0}].
\endaligned
\end{equation}
Finally, using \eqref{eqn:expand3},
\begin{equation*}
\aligned
     \varphi_{11}& =   [\gamma,\bar \gamma] _0+ [\gamma,\bar \gamma] _+  \\
                 &=    [\gamma^{-1,1},\bar \gamma^{-1,1}] _0 +  [\gamma^{-1,1},\bar \gamma^{-1,0}],
                \endaligned
\end{equation*}
so that
\begin{equation} \label{eqn:gamma11}
[ \varphi_{11},Y]=  -   [\gamma^{-1,1},\bar \gamma^{-1,0}].
\end{equation}
Combining Eqns.~\eqref{eqn:dbardgamma}--\eqref{eqn:gamma11}, we have:
\begin{equation}
    \left.\frac{\del^2}{\del  s \del\bar s} d(s) \cdot Y
     \right|_0
     = [\gamma^{-1,1},\bar\gamma^{-1,0}] + [\gamma^{-1,0},\bar\gamma^{-1,0}] +
        [\bar\gamma^{-1,1},\gamma^{-1,0}] .
\end{equation}
 The result then follows from \eqref{eqn:delta-laplacian}. \qed\endproof

\par So far, we have not assumed anything special about the biextension
variation $\FF$.    Of special interest in connection with the
Hodge conjecture is the case where the two \emph{normal functions}
 appearing in \eqref{eqn:biextension} are self-dual with respect
to the polarization $Q$ on $H := \gr ^W_{-1}$.



\begin{thm} \label{plurisubharmonicity} Let $h$ be the Hodge metric on $ \gr^W_{-1}$ and let  $\FF$ be  a self-dual biextension
over a curve $S$  with local normal form at a disk $(\Delta,s)$ at $s_0\in S$ given by $F(s)=e^{\Gamma(s)}$. Choose a  lift  $e(0) \in I^{0,0}_F$  of $1\in \Z(0)$  and let
\[
\gamma= \Gamma'(0) \in \End(W_0)_\C,\quad  t:= \gamma^{-1,0} (e(0))\in I^{-1,0}_F,
\]
where $\gamma^{-1,0}$ is the  Hodge component of type $(-1,0)$ of $\Gamma'(0)$.
Let  $\nu\in\Ext^1_{\VMHS}(\Z(0),\gr ^W_{-1} \FF)$ and its dual be the two normal functions associated to the biextension
and let $\delta(s)$ be the Deligne $\delta$-splitting of $\FF_s$. Then
\begin{enumerate}
\item the value of the infinitesimal invariant $\del\nu$ for the normal function $\nu$ at $s_0$ can be identified with $t$.
\item \begin{equation}
      \left.\frac{\del^2}{\del s \del \bar s} \delta (s)\right|_0   (e(0)) = h(t,t)\in \R_{\ge 0} , \,\quad t=\gamma^{-1,0}(e(0)).
  \label{eqn:laplace-4}
\end{equation}
\item The Chern form of the Hodge metric is semi-positive.
\end{enumerate}
\end{thm}

\proof    1. The point here is that $\gamma^{1,0} \in \Hom(I^{0,0}_F, I^{-1,0}_F) $  is the
derivative at $s_0$ of the period map for  the normal function $\nu$ which, from the set-up gets identified with $t$.\\
3. Follows from Theorem~\ref{chernformofhm} and 2.\\
2.
Recall \eqref{eqn:biextrel}. We have
\begin{eqnarray*}
    \frac{1}{2 \ii}  [\gamma^{-1,0},\bar \gamma^{-1,0}]e(0)
      &=&  - \frac{1}{2 \ii}\left( \gamma^{-1,0}(\bar\gamma^{-1,0}(e(0))) - \bar\gamma^{-1,0}(\gamma^{-1,0}(e(0)))
      \right)\\
      &=&  - \frac{1}{2 \ii}\left( \gamma^{-1,0}(\bar t) - \overline{\gamma^{-1,0}(\overline{t})}\right)\\
      &=& - \im (\gamma^{-1,0}(\bar t) ).
\end{eqnarray*}
Next, we express self-duality.
Observe that the derivative of the period map of the dual extension $\nu^*$   can be expressed as a functional on $W_{-1}$: it is zero on $W_{-2}$ and self-duality means precisely that on $H^*=\Hom(H,\Z(1)) $ it  restricts to the functional\footnote{For simplicity we have discarded  the Tate twist.}
\begin{equation*}
       \beta =Q(s,-) \in H^*
       \mapsto  - Q(s, t)\in \C.
\end{equation*}
This formula  implies that, tracing  through the identifications,
 one has   $\gamma^{-1,0}(\bar t )= - Q(\bar t, \gamma^{-1,0}e(0)) =- Q(\bar t, t)= Q(t,\bar t)$ and hence:
\begin{eqnarray*}
       \frac{1}{2 \ii}    [\gamma^{-1,0},\bar \gamma^{-1,0}]e(0)
         &=&   - \im(  Q( t, \bar t )).
\end{eqnarray*}
Since $h(t,t)= Q(-\ii t,\bar t)= -\ii Q(t,\bar t)$  is real, we get
indeed $  \frac{1}{2\ii}  [\gamma^{-1,0},\bar \gamma^{-1,0}]e(0)=  h(t,t)\in \R$.
\qed\endproof
\begin{corr} \label{biextcoroll} If $\VV$ is a variation of biextension type  over a curve $S$ with
self-dual extension data, then   $\delta(s)$ is a subharmonic function which vanishes
exactly at the points $s\in S$ for which the infinitesimal invariants of the associated normal functions vanish.
\end{corr}

\section{Reductive Domains And Complex Structures}\label{sec:reductive}
In this section we consider special classifying domains: the reductive ones. Recall that  a  homogeneous space  $D=G/H$   with $G$ a real Lie-group acting from the  left  on $D$ is   \emph{reductive} if the Lie algebra   $\eh=\Lie H$ has a vector space complement $\eun$ which is $\ad{H}$-invariant:
\begin{equation}\label{eqn:Reductive}
\eg= \eh\oplus \eun,\quad  [\eh,\eun]\subset \eun.
\end{equation}
Note that  this implies that $\eun$  is the tangent space at  the canonical base point of $D=G/H$; moreover, the tangent bundle is the   $G$-equivariant bundle associated to the adjoint representation of $H$ on $\eun$.
\subsection{Domains for Pure Hodge Structures}
These are reductive: in this situation $\eun_\C:= \eun_+\oplus \eun_-$ (see \eqref{eqn:SplitEnd}) is the complexification of $\eun:= \eun_\C\cap \eg$ and this is the desired complement.

Let us recall from \cite[Chap. 12]{periodbook} how   the connection form for the metric connection (the one for the Hodge metric)  can be  obtained. Start with the Maurer-Cartan form $\omega_G$ on $G$. It is a  $\eg$-valued $1$-form on $G$. Decompose $\omega_G$ according to the reductive splitting. Then $\omega=\omega^\eh$, the $\eh$--valued part,  is a connection form for the principal bundle $p:G \to G/H=D$. Let  $\rho: H \to \gl E$ be  a  (differentiable)  representation and let  $[E]= G\times _\rho E$ be the associated vector bundle.  It has an  induced connection which can be described as follows. Locally over  any open $U\subset D$ over which $p$ has a section $s:U \to G$, the bundle $[E]$ gets trivialized and the corresponding connection form then is $s^*(\dot\rho\comp \omega)$, where $\dot\rho: \eh \to \End E$ is the derivative of $\rho$.
\par
In the special case where $E= T_FD$ this leads to a canonical connection $\nabla_D$ on the holomorphic tangent bundle of $D$. If $D$ is a period domain this canonical connection is the Chern connection for the Hodge metric.

From this description the curvature   can  then directly   be calculated:
\begin{*thm}[\protect{\cite[Cor. 11.3.16]{periodbook} }] Let $D$ be a period domain for pure polarized Hodge structures and let $\alpha,\beta\in \eun=T_FD$ . Then  $R_D\in A^{1,1}_D(\End \eun)$, the curvature form  of the canonical connection $\nabla_D$ on the holomorphic tangent bundle of $D$ evaluates at $F$ as:
\[
R_D(\alpha,\bar\beta)= - \ad{  [\alpha,\bar \beta ]} ^\eh.
\]
\end{*thm}

\begin{rmk}\label{CompatibleComplexStructure}
The  above proof for the pure case makes crucial use of the compatibility of the complex structure of $D$ and  reductive structure: First, one needs the complex structure coming from the inclusion $D=G/G^F \subset  \check D= G_\C/G^F_\C$ to see that the Maurer-Cartan form is the real part of a holomorphic form, the Maurer-Cartan form on $G_\C$ and hence $\omega$ is the real part of a holomorphic form. Next, one uses that the  complex structure $J$ on $\eun$ is such that  $\eun_\pm\subset \eun_\C$  is the  eigenspace for $J$ with eigenvalue $\pm\ii$ and one makes the identification
\[
T_F D=(\eun, J)\simeq \eun_-.
\]
In the mixed case there are situations where the domain is reductive, but the complex structure then does not behave as in the pure case, as we now show.
 \end{rmk}

\subsection{Differential Geometry of Reductive Domains}
Let    $D=G/V$ be a reductive homogeneous space and a choice $\eg=\eh\oplus\eun$ of a reductive splitting.
Let us recall some major results from \cite{nomi}.
The $G$-invariant connections on  $T(D)$ are in one two one correspondence to bilinear $\ad H$--invariant functions
\[
\alpha: \eun \times \eun \to \eun.
\]
A given such connection $\nabla$ corresponds to
\[
\alpha(X,Y):= \nabla_X \tilde Y,
\]
where $\tilde Y$ is the vector field on $D$ obtained from $Y\in T_o(D)$ by left $G$-translation ($o\in D$ is the coset of $1\in G$).
The Maurer-Cartan induced connection $\nabla^{D}$  on $T(D)$ is the one for which $\alpha$ is identically zero. In loc. cit. it is called the \emph{canonical affine connection of the second kind}.
\par
Suppose that we have a $V$--invariant metric $g$ on $\eun$. This gives $G$--equivariant metric on $D$, likewise denoted $g$.   By \cite[Theorem 13.1]{nomi} a $G$-invariant  connection $\nabla$  on $T(D)$ is metric with  respect  to $g$
if and only if
\begin{equation}
\label{eqn:GinvConnection}
\nabla_X \tilde Y= \half  [X,Y]^\eun +U(X,Y),
\end{equation}
where $U:\eun \times \eun \to \eun$ is the  $\R$--bilinear form which is determined by the formula
\begin{equation}
\label{eqn:U}
2g(U(X,Y),Z) = g([Z,X]^\eun,Y) + g(X,[Y,Z]^\eun).
\end{equation}
Moreover, the connection is free of torsion if and only if $U$ is a symmetric form.
For  the Maurer-Cartan induced connection  the left hand side of \eqref{eqn:GinvConnection} vanishes and so it is metric, precisely when
\begin{equation}
\label{eqn:U2}
U(X,Y)=-\half  [X,Y]^\eun .
\end{equation}
So this can only be without torsion if  $[X,Y]^\eun=0$. In fact,  By \cite[Theorem 10.3]{nomi} its torsion is  given by
\begin{equation}
\label{eqn:torsionformula}
T(X,Y)= - [X,Y]^\eun.
\end{equation}
So, the canonical connection in general differs from the Levi-Civita connection.
\begin{rmk}\label{conncomplextend}
1) We extend the above connections to the complex tangent bundle $T_\C(D)$. The same considerations then hold provided  $g$ and $U$ are replaced by their  $\C$--bilinear extensions.
\\
2) Note that in general only the thus extended canonical connection preserves the decomposition $T_\C (D)= T^{1,0}D\oplus T^{0,1}D$ into  the holomorphic and anti-holomorphic tangent  bundles. For the Levi-Civita connection this holds if the metric is K\"ahler.
\end{rmk}

\subsection{Split Domains}
Mixed domains are seldom  reductive, and, even if they are, we shall see that the complex structure  does not satisfy the compatibility required  by Remark~\ref{CompatibleComplexStructure}.

\begin{exmples} 1. Suppose  $\Lambda=0$. Then  equation \eqref{eqn:Double} implies that
 $\eun= \eun_\C\cap \eg_\R$ is the desired complement. Note that in the pure case this equals also $\eun_\C\cap \eg$. This difference will influence the curvature calculations.
Domains with $\Lambda=0$ are called \emph{split domains} because they parametrize split mixed Hodge structures.
  We investigate these below in more detail.
  \newline
2. We consider the general mixed situation. Let $D^{\rm split}$ be the subdomain of $D$ parametrizing split mixed Hodge structures\footnote{This has been called $D_\R$ in \S~2.}. This domain can be identified with $G_\R/ G^F_\R$, where $F$ is a fixed split mixed Hodge structure.
Note that  $\eun_\C\oplus \Lambda$ has a real structure which makes $D^{\rm split}$ a reductive domain for the splitting
\[
\eg_\R= \underbrace{\eg^{0,0}\cap \eg_\R}_{\Lie{G^F_\R}}  \oplus (\eun_\C\oplus \Lambda)_\R.
\]
In general $D^{\rm split}$ only has the structure of a differentiable manifold.
\newline
3.  In general the group $G_\R$ does not act transitively on $D$. But there is another natural subgroup of $G$ which does act transitively. To explain this, introduce (for $r<0$):
\begin{eqnarray*}
G^W_r & :=  & \sett{g\in G}{ \text{ for all $k$ the restriction }  g| (W_k/W_{k+r}) \text{ is real}.}
\end{eqnarray*}
Note that $G^W_{-2}$ contains $\exp(\Lambda)$ as well as $G_\R$ and hence it  acts transitively on $D$. Under  the minimal condition
\[
\Lie{G^W_{-2}}=\eg_\R\oplus \ii \Lambda
\]
we clearly get a reductive splitting
\[
\Lie{G^W_{-2}} = \eg^{0,0}\cap \eg_\R \oplus \left[ (\eun_\C\oplus \Lambda)_\R\oplus \ii \Lambda_\R \right].
\]
Domains which satisfy this condition are called \emph{close to splitting}.
An example is provided by the so-called type II domains from \cite{sl2anddeg}.

Note that in general $(\eun_\C\oplus \Lambda)_\R$ does not admit  a complex structure: $\dim\Lambda$ can be odd! 
\end{exmples}

\subsection{Two Step Filtrations} \label{ssec:TwoStep}
This case has been treated in detail in \cite[\S~2]{usui}. 
The domains in question  are  examples of split domains, and hence they are reductive.
The mixed Hodge structures  they parametrize indeed split over $\R$ since the associated weight filtration has only two consecutive steps, say
 $0=W_0\subset W_1\subset W_2=H$.

Assume that we are given  two polarizations on $W_1$ and $\gr^W_2$, both denoted $Q$. One can choose an adapted (real) basis for $H$  which
\begin{itemize}
\item restricts to a $Q$--symplectic basis $(a_1,\dots,a_g,b_1,\dots,b_g)$ for $W_1$;
\item the remainder of the basis  $(c_1,\dots,c_k,c'_1,\dots,c'_k,d_1,\dots,d_\ell)$ projects to a basis for $\gr^W_2$ diagonalizing  $Q$, i.e. $Q=$diag$(-\mathbf{1}_{2k},\mathbf{1}_\ell)$.
\end{itemize}
Then
\[
G= \sett{
\begin{pmatrix}
A & B \\
0 & C\end{pmatrix}}{A\in \smpl{g;\R}, \, C\in \ogr{2k,\ell}, \, B\in \C^{2g\times (2k+\ell)}},
\]
reflecting the Levi-decomposition. More invariantly, the two matrices $A$ and $C$ on the diagonal give the semi-simple part $G^{\rm ss}$ while the matrices $B$ give the unipotent radical
\[
G^{\rm un}\simeq \Hom_\C( \gr^W_2,W_1).
\]
Here, the isomorphism (via the exponential map) in fact identifies $G^{\rm un}$ with its Lie-algebra:
\begin{equation} \label{eqn:UnPart}
\eg^{\rm un}=  \Hom_\C( \gr^W_2,W_1),
\end{equation}
the endomorphisms in $\eg$ which lower the weight by one step.
\par
The real group  $G_\R$ consists of the  group given by similar matrices, except that  now the matrices $B$ are taken to be real. In particular
\begin{equation} \label{eqn:UnPar2t}
\eg^{\rm un}_\R= \eg^{\rm un}\cap \eg_\R =  \Hom_\R( \gr^W_2,W_1).
\end{equation}

Next, fix  the    Hodge flag $F=\set{F^2\subset F^1\subset F^0=H_\C}$  which has the following  adapted unitary basis
\begin{equation}
\label{eqn:basisHodgeflag}
\left.\aligned
  & \hspace{-2em}(\underbrace{\underbrace{f_1,\dots,f_k}_{F^2},  d_1,\dots,d_\ell,  f'_1,\dots,f'_g}_{F^1},  \bar f'_1,\dots, \bar f'_g,\bar f_1,\dots,\bar f_k), \\
f_k:&= \frac{1}{\sqrt 2} [c_k- \ii c'_k] , \quad   f'_k= \frac{1}{\sqrt 2} [a_k- \ii b_k].
\endaligned \right\}
\end{equation}
The group $G^F$ consists of the subgroup of $G$ with $A=\begin{pmatrix} U & -V\\ V& U\end{pmatrix}$, $U+\ii V\in \ugr g$, $C\in \ogr{2k}\times\ogr{\ell}$ and  the matrices $B$  are of the form
\begin{equation*} 
\sett{\begin{pmatrix} B' \\ -\ii B' \end{pmatrix}}{B'\in \C^{g\times (2k+\ell)}}.
\end{equation*}
Note that $G^{\rm ss}/ G^F\cap G^{\rm ss}=D_1\times D_2$, the product of the domain $D_1\simeq \mathbf{H}_g$, parametrizing weight $1$ Hodge polarized structures with $h^{1,0}=g$ and $D_2$ parametrizing weight $2$ polarized Hodge structures with $h^{2,0}=k, h^{1,1}=\ell$. The
natural projection
\begin{equation}
\label{eqn:ComplexBundle}
G/G^F  \to G^{\rm ss}/ G^F\cap G^{\rm ss}=D_1\times D_2
\end{equation}
is a  holomorphic  bundle with fiber  associated to the adjoint representation of  $G^{\rm ss}\cap\nobreak G^F$ on $\eg^{\rm un}/\eg^{\rm un}\cap \eg^F$. Explicitly, this action is
\[
g \cdot [B] = [ABC^{-1}] ,\quad g= \begin{pmatrix}
A & 0 \\
0 & C\end{pmatrix} .
\]
The fiber of \eqref{eqn:ComplexBundle} over $F$ is the affine space consisting of  the extension data  of $(W_1, F)$ by $(\gr^W_2,F)$ on which $G^{\rm un}$ acts transitively as  the  group of translations.
The group $G^{\rm ss}$ acts on this fiber bundle by holomorphic transformations from the left:  $g\in  G^{\rm ss}$  sends the fiber over $F$ biholomorphically to the fiber over $g\cdot F$.

To obtain  a reductive decomposition  $ \eg_\R  = \eh \oplus \eun$,  set
\begin{equation}
\label{eqn:redcomp}
\eh:= \eg^{0,0}\cap \eg ,\quad \eun= \eun^{\rm ss}\oplus \eg^{\rm un}_\R ,\quad
 \eun^{\rm ss}=\left(  \oplus_{p\not=0}  \eg^{p,-p} \right) \cap \eg.
\end{equation}
Let us study the metric properties of the Hodge metric $h$ and its Chern connection $\nabla_h$. It is invariant under the Hodge metric and so is determined by Eqn.~\eqref{eqn:GinvConnection}.
\begin{lemma} The canonical connection $\nabla^D$ on the  complex tangent bundle $T_\C(D)$ of  $D=G/G^F$ given by the reductive decomposition \eqref{eqn:redcomp} is distinct from the  (extended)  Chern connection $\nabla_h$ on $T_\C(D)$.
\end{lemma}
\proof
Both  connections are metric  for the Hodge metric and so they are both given by the formula \eqref{eqn:GinvConnection}.
In particular,  for $X,Y\in \eg_\C$  we have
\[
\nabla^D_X\, \tilde Y =  U(X,Y).
\]
Let us  calculate   $U(X, X)$, $X\in \eg^{-1,0}$ with the aid of  \eqref{eqn:U} where (cf. Remark~\ref{conncomplextend}) $g$ is the complex bilinear extension of the real part of the Hodge metric on  $\eg_\C$.  We then see that for $Z\in \eg^{-1,1}\oplus \eg^{1,-1}$, we get
\[
\aligned
2g(U(X,  X), Z) &= g([Z,X],  X)+ g(X,[\bar X,Z])\\
&= h([Z,X],  \bar X)+ h([\bar X,Z],\bar X)\\
                              &= - h (Z,[X^*,\bar X])+  h (Z,[\bar X^*,X])\\
                              &=  g(Z, [X, \bar X^*]- [\bar X, X^*])
                              \endaligned
                              \]
 where the third line follows from \eqref{eqn:MetConv}. Hence $U(X,X)=  \half  ( [X, \bar X^*]- [\bar X, X^*] )$ which  does not always vanish.
 Indeed in the  basis \eqref{eqn:basisHodgeflag} the tangent vector $X$ corresponds to a matrix with $A=C=0$ and $B$ arbitrary,  while $\bar X^*$ is the transpose conjugate so that
 $U(X,X)= \begin{pmatrix}
   \text{Im } B \Tr B & 0\\
 0 &  -\text{Im  }\Tr B B
 \end{pmatrix}$.
  Now compare this with what happens for  $\nabla_h$.  Eqn.~\eqref{eqn:U2} tells us that we must have  $U(\nabla^h) (X, X)= - \half [X,X]=0$.    Indeed, the canonical connection has $\nabla^D_X\,\tilde Y =\alpha(X,Y)=0$. This shows that $\nabla^D\not= \nabla^h$.
 \qed\endproof

As to the complex structure we have:
\begin{lemma}
The  complex structure compatible with the reductive structure   is not the one coming from the embedding $G/G^F\subset G_\C/G_\C^F$.
\end{lemma}
\proof
Write
\[
\aligned
\eg^{\rm un} &= \underbrace{\eg^{0,-1}_F\oplus \eg^{1,-2}_F}_{\eg^{\rm un}_{F,+}}  \oplus  \underbrace{ \eg^{-1,0}_F\oplus \eg^{-2,1}_F}_{\eg^{\rm un}_{F,-}}\\
\eun^{\rm ss} &= [\underbrace{ \eg^{-2,2} \oplus \eg^{-1,1} }_{\eun^{\rm ss}_{F,-}} \oplus\underbrace{ \eg^{1,-1}\oplus \eg^{2,-2}}_{{\eun^{\rm ss}_{F,+}}}]\cap \eg.
\endaligned
\]
Since  $\eg^{\rm un}_{F,+ }=\eg^F\cap \eg^{\rm un}_F$, the tangent space at $F$ to
\[
D^{\rm un}:= G^{\rm un}/G^F\cap G^{\rm un}
\]
gets identified with
\[
\aligned
T_F D^{\rm un}
          &= \eg^{\rm un} /  \eg^{\rm un}_{F,+} \\
          &=   \eg^{\rm un}_{F,-},
\endaligned
\]
a space of complex dimension $g(2k+\ell)$.  The complex structure comes from the standard complex structure $J$
on  $\eg^{\rm un}$,  since $T_F D^{\rm un} $ is a quotient thereof.\footnote{I.e.,  $J$ is multiplication by $\ii$.}
Next, note that
\[
\eg^F_\R\cap \eg^{\rm un}= 0
\]
and so
\[
G^{\rm un}_\R/G^F\cap G^{\rm un}_\R=\eg^{\rm un}_\R=\Hom_\R( \gr^W_2,W_1)
\]
and this space gets a complex structure  thanks to the weight one  Hodge structure induced by $F$ on $W_1$.
It   is induced by  a complex structure  $J_1^F$  whose complexification on $\eg^{\rm un}$ has eigenvalues as in the following table:
\begin{center}
\begin{tabular}{c|c|c|c|}
&   $I^{2,0}_F$ &  $I^{1,1}_F$  & $I^{0,2}_F$ \\
\hline
$I^{1,0}_F$ &  $\ii $ & $ \ii $ &  $ \ii $\\
\hline
$I^{0,1}_F$ &    $-\ii $ & $ -\ii $ &  $- \ii $\\
\hline
\end{tabular}
\end{center}
One deduces that the complex structure $(\eg^{\rm un}_\R,J_1^F)$ is not isomorphic to the complex structure $(\eg^{\rm un}_\R,J^F)$.

The complex structure  $J^F$ coming from $G/G^F\subset G_\C/G_\C^F$ identifies the holomorphic tangent space at $F$ as follows:
 \[
 T_FD=  \eg/\eg^F =  \eun^{\rm ss}_{F,-} \oplus \eg^{\rm un}_{F,-} \simeq   (\eun^{\rm ss}, J^F)\oplus ( \eg^{\rm un}_\R,J^F).
 \]
 The natural complex structure $J^F$ on $\eun^{\rm ss}$  comes from the one inducing the complex structure on the base $D_1\times D_2$ of the fiber bundle \eqref{eqn:ComplexBundle}.

 Taking the same complex structure on $\eun^{\rm ss}$ but the other on $\eg^{\rm un}_\R$  leads to a different holomorphic tangent space
 \[
 (T_FD, J_1^F)= (\eun^{\rm ss}, J^F)\oplus ( \eg^{\rm un}_\R,J_1^F);
\]
it is a complex structure  on $\eun$ whose $\pm\ii$--eigenspaces inside $\eun\otimes \C$  are given by
\[
\eun_{F,+} = \eun^{\rm ss}_{+,F}  \oplus \Hom_\C(\gr^W_2\otimes\C , I^{0,1}_F)
\]
respectively
\[
\eun_{F,-} = \eun^{\rm ss}_{-,F}  \oplus \Hom_\C(\gr^W_2\otimes\C , I^{1,0}_F).
\]
Finally, note that the isomorphism
\[
(T_F D, J_1^F)  =  (\eun,J_1^F) \simeq \eun_{F,-}.
\]
gives $T_FD$ the  complex structure which is required in the standard curvature calculations  for reductive domains, as explained above. However, as we have seen, this structure is not the one which comes from the embedding $D=G/G^F\into G_\C/G_\C^F$.

\qed\endproof
\begin{rmq}  1. Clearly,   $J_1^F$ and $J^F$ commute. 
\changed{}{More can be said. For a fixed filtration $F$, the group  of extension data, $\Ext (\gr^2W,W_1)$ can be identified with
the intermediate Jacobian $\mathcal J_F$ of the weight-$(-1)$ Hodge structure $\Hom (\gr_2^W, W_1)$. It carries two canonical complex structures:
the one by Weil, defined by the action of the Weil-operator, and the one by Griffiths, having the property that the $\pm \ii$-eigenspaces
are given by ${F'}^{-1}$ and its conjugate, where $F'$ is the induced Hodge filtration on $\Hom (\gr_2^W, W_1)$.
From this description one sees that $J^F$ gives Weil's intermediate Jacobian, while $J^{F}_1$ gives Griffiths Jacobian.  With
the latter structure the  family $\sett{\mathcal J_F }{ [F]\in D_1\times D_2 }$ is indeed varying holomorphically.}
\\
2. Consider the surjective morphism
\[
 G_\R/G^F_\R   \to G^{\rm ss} / G^F_\R\cap G^{\rm ss}=D_1\times D_2.
\]
It is a real-analytic complex vector bundle associated to the  $G^{\rm ss}\cap G^F$--representation space $\eg^{\rm un}_\R$.
This is also a \changed{}{$J_1$-}holomorphic fiber bundle: if $U \in \ugr{g}$ and $V\in [ \ogr{2k}\times \ogr{\ell}] $, the action on $\varphi\in  \Hom_\R(\gr^W_2,W_1)$ is given by $\varphi\mapsto U\comp \varphi\comp V^{-1}$ and hence is $J_1^F$-complex. However, the action of $G^{\rm ss}$ on this bundle is no longer holomorphic:  $g=(U,V) \in \smpl{g}\times \ogr{2k,\ell}$ sends $\varphi$ in the fiber over $F$ to  $U\comp \varphi$ in the fiber over $g\cdot F$ and since $U$ and $J_1^F$ only commute when   $U\in \ugr{g}$ this is \emph{not} a $J_1^F$-complex-linear isomorphism. Since in our situation $G/G^F\simeq G_\R/G^F_\R$, this also confirms that the two complex structures are distinct.
\end{rmq}

\appendix

\section{The Levi-Civita Connection} \label{Appendix A}
Suppose  that $M$ is a complex manifold and $\frak X^{1,0}_M$
and $\frak X^{0,1}_M$ denote the sheaves of complex vector fields of type $(1,0)$ and $(0,1)$ respectively.
Then, the conjugation action $u\mapsto u_c$ defined by
$$
      u_c \cdot f = \overline{u \cdot \overline f}
$$
defines an isomorphism of sheaves $\frak X^{1,0}_M\mapright{\sim}  \frak X^{0,1}_M$ as modules over the sheaf $C^{\infty}(M,\Bbb R)$
of real valued smooth functions on $M$.  It  restricts to a conjugate linear morphism   between the sheaves
of holomorphic and anti-holomorphic vector fields on $M$.

\begin{lemma} Let $\frak X_M$ denote the sheaf of $C^{\infty}$ real vector fields on $M$.
Then,
$$
\begin{matrix}
     \frak X^{1,0}_M &\to & \frak X_M \\
      u &\mapsto&  u_r := u + u_c
\end{matrix}
$$
defines a linear isomorphism over $C^{\infty}(M,\R)$. Moreover, if $x$ and $y$ are  holomorphic vector fields, then
\begin{equation}\label{eqn:tag2}
        [x_r,y_r] = [x,y]_r    .
\end{equation}
\end{lemma}
\proof
If $z_j = x_j + \sqrt{-1}y_j$ is a system of holomorphic coordinates on an open subset $U$  of $M$ then
\begin{equation*}
       \left(\frac{\partial}{\partial z_j}\right)_r = \frac{\partial}{\del x_j},\qquad
        \left(\sqrt{-1}\frac{\partial}{\partial z_j}\right)_r = \frac{\del}{\del y_j}
\end{equation*}
and hence  the  stated  morphism induces  an isomorphism over any holomorphic coordinate chart.  Using partitions of unity, it then follows that  it is a global isomorphism, $\frak X^{1,0}_M\mapright{\cong} \frak X_M$.  Since holomorphic and anti-holomorphic
vector fields commute, \eqref{eqn:tag2} follows.\qed\endproof

Let $g$ be a Riemannian metric on the underlying $C^{\infty}$-manifold of $M$.  Then, the associated
Levi-Civita connection $\nabla^{\text {LC}}$ is determined by the Koszul formula:
\begin{equation}\label{eqn:3}
\aligned
       2g(\nabla_X^{\text {LC}}\, Y,Z) = Xg(Y,Z) &+ Yg(X,Z) -Zg(X,Y)  \\
                                    &+ g([X,Y],Z) - g([X,Z],Y) - g([Y,Z],X).
\endaligned     															\end{equation}
In particular, if $h$ is a hermitian metric on $M$ given as a pairing of sections of $\frak X^{1,0}_M$ we
obtain an associated Riemannian pairing on sections of $\frak X_M$ by the rule
\begin{equation}\label{eqn:4}
 g(u_r,v_r) = \text{Re}\, h(u,v).
\end{equation}

\par By the above remarks, in order determine the Levi-Civita connection of the metric \eqref{eqn:4} it is
sufficient to evaluate the expression \eqref{eqn:3}  on vector fields $X = x_r$, $Y=y_r$ and $Z = z_r$ with
$x$, $y$ and $z$ holomorphic vector fields on $M$.  Unraveling the above, for holomorphic vector
fields $u$, $v$ and $w$ we have:
\begin{equation} \label{eqn:5}
\left.\aligned
     w_r \cdot  g(u_r,v_r) &= w\cdot \text{Re}\, h(u,v) + w_c \cdot \text{Re}\, h(u,v) \\
                            &= (1/2)w\cdot(h(u,v)+h(v,u)) + (1/2)\overline{w\cdot(h(u,v)+h(v,u))} \\
                            &= \text{Re}\, (w\cdot(h(u,v)+h(v,u))).
\endaligned \right\}.
\end{equation}

\begin{lemma} The Levi-Civita connection $\nabla^{\text {LC}}$ of the Riemannian metric \eqref{eqn:4}  underlying a hermitian metric
$h$ on a complex manifold $M$ is determined by the formula:
\[
\aligned
     2g(\nabla_{x_r}^{\text {\rm LC}}\, y_r,z_r)
     &=\text{\rm Re}\left( x\cdot(h(y,z)+h(z,y)) + y\cdot(h(x,z)+h(z,x)) \right. \\
                      &\left.\quad - z\cdot(h(x,y)+h(y,x))
                                    +h([x,y],z) - h([x,z],y) - h([y,z],x)\right),
\endaligned
\]
where $x_r$, $y_r$ and $z_r$ arise from underlying holomorphic vector fields $x$, $y$ and $z$.
\end{lemma}
\proof
 The right hand side of the Koszul formula \eqref{eqn:3} for the Levi--Civita connection is  the sum of the terms
\[
\aligned
    x_r\cdot g(& y_r,z_r) + y_r\cdot g(x_r,z_r) - z_r\cdot g(x_r,y_r) \\
     &=\text{Re}\,\left( x\cdot(h(y,z) +h(z,y))   
      + y\cdot(h(x,z)+h(z,x))  
      -z\cdot(h(x,y)+h(y,x))\right)   
\endaligned
\]
and
$$
\aligned
    g([x_r,y_r],z_r) &-g([x_r,z_r],y_r) - g([y_r,z_r],x_r) \\
    &=\text{Re}\, \left( h([x,y],z) - h([x,z],y) - h([y,z],x)\right). \hfill \qed
\endaligned
$$
\endproof

\par We want to apply this formula in the case of the mixed Hodge metric and holomorphic
vector fields of the form
$$
             \tilde\alpha(e^u \cdot F) = L_{e^u*}\alpha
$$
where $\alpha\in\frak q$ acts as the derivation
$$
              \alpha \cdot f = \left.\frac{d}{dz}f(e^{z\alpha}\cdot F)\right|_{z=0}
$$
on germs of functions at $F$ and $u\mapsto e^u \cdot F$ gives a biholomorphism from a neighborhood
of $0$ in $\frak q$ to a neighborhood of $F$ in $D$.

\begin{lemma} Let $\alpha,\beta,\gamma\in\frak q$.  Then\footnote{compare with Cor.~\ref{corr:H-form}},
$$
      \left.\tilde\alpha \cdot h(\tilde\beta,\tilde \gamma)\right|_F
       =  -h_F(\beta,[\pi_+(\bar\alpha),\gamma]).
$$
\end{lemma}
\proof We have
$$
      \left.\tilde\alpha \cdot h(\tilde\beta,\tilde \gamma)\right|_F
      = \left.\frac{d}{dz} h_{e^{z\alpha}\cdot F}(\tilde\beta,\tilde\gamma)\right|_{z=0}
      = \left.\frac{d}{dz} h_F( L_{f(z)*}\beta,L_{f(z)*}\gamma)\right|_{z=0} ,
$$
where $f(z) = \exp(-\bar z\pi_+(\bar\alpha) + O^2(z,\bar z))$.  Therefore,
$$
      \left.\tilde\alpha \cdot h(\tilde\beta,\tilde \gamma)\right|_F
       =  -h_F(\beta,[\pi_+(\bar\alpha),\gamma]). \hfill\qed
$$
\endproof

\begin{thm} \label{AppendixThm} For $x_r$, $y_r$ and $z_r$ arising from $\tilde x$, $\tilde y$, $\tilde z$ we have
$$
\aligned
     2g(\nabla_{x_r}^{\text {\rm LC}}\, y_r, z_r)
     =  &-\text{\rm Re}(h_F(y,[\pi_+(\bar x),z]) + h_F(z,[\pi_+(\bar x),y])) \\
        &-\text{\rm Re}(h_F(x,[\pi_+(\bar y),z]) + h_F(z,[\pi_+(\bar y),x])) \\
        &+\text{\rm Re}(h_F(x,[\pi_+(\bar z),y]) + h_F(y,[\pi_+(\bar z),x])) \\
        &+\text{\rm Re}(h_F([x,y]-[x^*,y]-[y^*,x],z)).
\endaligned
$$
\end{thm}

\begin{corr}  \label{LC1} If $\tilde x$ and $\tilde y$ arise from
$x,y\in \, \frak g^{-p,-q}_{(F,W)}$ , $p,q>0 $  by left translation,
then for the corresponding vector fields $x_r,y_r$ we have
$$
      \nabla_{x_r}^{\text {\rm LC}}\, y_r = \half [x,y]_r.
$$
\end{corr}
\proof  Let $z_r$ arise from $\tilde z$ as above.   The first
two lines in the formula of Theorem~\ref{AppendixThm}  vanish since $\pi_+(\bar x)=
\pi_+(\bar y)=0$ because $x,y\in\Lambda_F$.  As for the third line of the formula for $\nabla$,
we note that $\pi_+(\bar z)$ can never have a component of type $(0,0)$
and hence $[\pi_+(\bar z),x]$ is orthogonal to $y$ and
$[\pi_+(\bar z),y]$ is orthogonal to $x$.  So, only the last line
of the formula of Theorem~\ref{AppendixThm}  survives
which gives
$$
     2\nabla_{x_r}^{\text {LC}}\, y_r = [x,y]_r - \pi_\eq([x^*,y]+[y^*,x])_r.
$$
The last term then vanishes since $[x^*,y]+[y^*,x]$ has type $(0,0)$.\qed
\endproof

\begin{lemma}\label{auxlemma}
Let $x,y,z \in \eq_F$. Put $t:= [y^*,x]+[x^*,y]$. Then
\[
\re {h_F} (t, \pi_+ (\bar z)) =\re  {h_F} (\overline{\pi_+(t)^*}^{\, *},z).
\]
If $(F,W)$ is split over $\R$ then
\[
\overline{\pi_+(t)^*}^{\, *}=\pi_- (\bar t)= \pi_-([\bar y^*,\bar x]+[\bar x^*,\bar y]).
\]
\end{lemma}
\proof Since $h_F(u,v)= \overline{h_F(v,u)}$ we have
\[
\aligned
\re{h_F} (t,\pi_+(\bar z )))&= \re{h_F} (\pi_+(\bar z),t)\\
&=\re{h_F} ( \bar z,\pi_+(t))\\
&=\re{}  \tr ( \bar z \comp (\pi_+(t))^*)\\
&= \re{} \tr (z\comp \overline{\pi_+(t))^*})\\
&= \re{h_F}  ( z,  \overline{\pi_+(t))^*}\,^*)\\
&=\re{h_F} ( \overline{\pi_+(t))^*}\,^*,z)  . \hspace{4em}
\endaligned
\]
In the split case, $*$ and complex conjugation commute and hence $ \overline{\pi_+(t)^*}^{\, *}=\overline {\pi_+ t}= \pi_-(\overline t)$ and the second assertion follows.
\qed\endproof

\begin{thm}  \label{ThmAppendA2} If $(F,W)$ is split over $\Bbb R$ then $2\nabla_{x_r}^{\text {\rm LC}}\, y_r$ at $F$ is
the real derivation defined by
$$
\aligned
         &-\pi_\eq([\pi_+(\bar x)^*,y] + [\pi_+(\bar x),y] +[\pi_+(\bar y)^*,x] + [\pi_+(\bar y),x])  \\
       &  \quad +\pi_-([\bar y^*,\bar x]+[\bar x^*,\bar y]) + \pi_\eq([x,y]- [x^*,y]-[y^*,x]).
\endaligned
$$
\end{thm}
\proof   Applying Lemma~\ref{auxlemma} to Theorem~\ref{AppendixThm}  we have
$$
\aligned
     2g(\nabla_{x_r}^{\text {LC}}\, y_r, z_r)
     =  &-\text{\rm Re}(h_F(y,[\pi_+(\bar x),z]) + h_F(z,[\pi_+(\bar x),y])) \\
        &-\text{\rm Re}(h_F(x,[\pi_+(\bar y),z]) + h_F(z,[\pi_+(\bar y),x])) \\
        &+\text{\rm Re}(h_F(\pi_-([\bar y^*,\bar x]),z) + h_F(\pi_-([\bar x^*,\bar y]),z)) \\
        &+\text{\rm Re}(h_F([x,y],z) - h_F([x,z],y) -h_F([y,z],x))
\endaligned
$$
which becomes
$$
\aligned
     2g(\nabla_{x_r}^{\text {LC}}\, y_r, z_r)
     =  &-\text{\rm Re}(h_F([\pi_+(\bar x)^*,y] + [\pi_+(\bar x),y],z) \\
        &-\text{\rm Re}(h_F([\pi_+(\bar y)^*,x] + [\pi_+(\bar y),x],z) \\
        &+\text{\rm Re}(h_F(\pi_-([\bar y^*,\bar x]) + \pi_-([\bar x^*,\bar y]),z) \\
        &+\text{\rm Re}(h_F([x,y]- [x^*,y]-[y^*,x],z)). \hfill\qed
\endaligned
$$
\endproof

\begin{corr}   \label{LC2}Assume $(F,W)$ is split over $\R$.  Let $x_r$ and $y_r$  be vector fields
arising from $\tilde x$, $\tilde y$ with $x$ and $y$ of type $(-1,1)$.  Then,
$$
         \nabla_{x_r} ^{\text {\rm LC}}y_r = \half [x,y]_r.
$$
%
\end{corr}
\proof  In this case, by Lemma~\ref{lemma:adjoint} $x^*=\bar x$, $y^*=\bar y$ and so $\pi_+(\bar x)^* = x$ and $\pi_+(\bar y)^* = y$.
We also note that $[\bar x,y]$ and $[\bar y,x]$ project to zero in $\eq=\eun_-\oplus\Lambda$.  Therefore, the formula of
Theorem~\ref{ThmAppendA2}  reduces to the stated form.\qed
\endproof

\begin{corr} Assume that $W$ has only two weight graded quotients which are
adjacent and let $x_r$ and $y_r$ arise from $\tilde x$ and $\tilde y$ with $x$
and $y$ of type $(-1,0)$.  Then,
$$
        2\nabla_{x_r}^{\text {\rm LC}}\, y_r = -[\bar x^*,y]_r - [\bar y^*,x]_r
$$
\end{corr}
\proof For $u$ and $v$ of type $(-1,0)$ in this setting we have $\pi_+(\bar u) = \bar u$
and $[u,v] = [\bar u,v]=0$.   Likewise, $[v^*,u]$ and $[\bar v^*,\bar u]$ are type $(0,0)$ while $[\bar v^*,u]$
is type $(-1,1)$.  Consequently, the formula of Theorem~\ref{ThmAppendA2} reduces to the stated form.\qed
\endproof

Let us apply this to flow curves
\[
\gamma_x:   \quad  t\longmapsto \exp( tx) \cdot F,\quad x\in\eq_F
\]
and set
\[
x(t_0):= \left.\frac{d \gamma_x }{dt}\right|_{t_0} \in \eq.
\]
\begin{corr}\label{LC4} {\rm (1)} For $  x,y\in \Lambda_F$ of the same type we have
 \[
 \nabla^{\rm LC} _{x(t)_r}  y(t)_r= \half [ x(t), y(t)]_r.
 \]
 {\rm (2)} The flow  curve $\gamma_x$ is a geodesic. This is in particular the case when $\gamma_x$ is the image under a period map.\\
 {\rm (3)} Suppose that $x,y,z \in \Lambda_F$ have the same type and commute. Then the Riemann curvature
 \[
R(x_r,y_r)z_r = \nabla_x^{\text {\rm LC}} \nabla_{y_r}^{\text {\rm LC}} z_r -\nabla_{y_r}^{\text {\rm LC}} \nabla_{x_r}^{\text {\rm LC}} z_r - \nabla_{[x_r,y_r]}^{\text {\rm LC}}z_r
\]
vanishes.
\end{corr}
\proof Under the flow the type need not be preserved.  However, an application of Lemma~\ref{higgs} shows that the types are preserves when we start with   $x  \in \eg^{-p,-q}_F$ with $p,q>0$.  Then (1) follows from Cor.~\ref{LC1}.
In particular, this vanishes for  $x=y$. By definition the curve $\gamma_x$ then is a geodesic. The formula for the Riemann curvature   implies (3). \qed\endproof

\end{document}